\documentclass[12pt,reqno]{amsart}
\usepackage{graphicx}
\usepackage{amsmath, amsthm, amssymb, stmaryrd}
\usepackage{hyperref}
\usepackage{enumerate}
\usepackage{url}
\usepackage{scalerel}

\usepackage{mathtools}

\usepackage{multirow, array}
\usepackage{placeins}

\usepackage{caption} 
\advance \topmargin by -\headheight
\advance \topmargin by -\headsep
     
\evensidemargin \oddsidemargin
\newtheorem{thm}{Theorem}[section]
\newtheorem{lemma}[thm]{Lemma}

\newtheorem{cor}[thm]{Corollary}
\newtheorem{conj}[thm]{Conjecture}

\theoremstyle{definition}
\newtheorem{defn}[thm]{Definition}

\theoremstyle{remark}
\newtheorem{remark}[thm]{Remark}

\numberwithin{equation}{section}

\newcommand*\wrapletters[1]{\wr@pletters#1\@nil}
\def\wr@pletters#1#2\@nil{#1\allowbreak\if&#2&\else\wr@pletters#2\@nil\fi}

\def\eps{\varepsilon}

\def\le{\leqslant} \def\ge{\geqslant}

\def \bN {\mathbb N}

\def \bR {\mathbb R}
\def \bZ {\mathbb Z}

	\def \bF {\mathbf F}

\def \bk {\mathbf k}

\def \dag {\dagger}

\DeclareMathOperator*{\bigdelta}{\scalerel*{\mathrm{\Delta}}{\sum}}

\title[On the Extended 1-2-3 Conjecture of Pilz]{On the Extended 1-2-3 Conjecture of Pilz}

\author[Philippa Holdridge]{Philippa Holdridge}
\address{HUN-REN Alfr\'ed R\'enyi Institute of Mathematics, Re\'altanoda utca 13--15., H-1053 Budapest,  Hungary;   \newline \hspace*{4mm}
MTA--HUN-REN RI Lend\"ulet ``Momentum'' Arithmetic Combinatorics Research Group, Re\'altanoda utca 13--15., H-1053 Budapest,  Hungary}
\email{holdridge.philippa@renyi.hu}

\author{P\'eter P\'al Pach}

\address{HUN-REN Alfr\'ed R\'enyi Institute of Mathematics, Re\'altanoda utca 13--15., H-1053 Budapest,  Hungary;   \newline \hspace*{4mm}
MTA--HUN-REN RI Lend\"ulet ``Momentum'' Arithmetic Combinatorics Research Group, Re\'altanoda utca 13--15., H-1053 Budapest,  Hungary; \newline \hspace*{4mm} 
Department of Computer Science and Information Theory, Budapest University of Technology and Economics, M\H{u}egyetem rkp. 3., H-1111 Budapest, Hungary; \newline \hspace*{4mm}
Extremal Combinatorics and Probability Group (ECOPRO), Institute for Basic Science (IBS), Daejeon, South Korea.}
  
  \email{pachpp@renyi.hu}

\subjclass[2020]{}
\thanks{}
\date{}

\begin{document}

\begin{abstract}
We resolve (for all sufficiently large $n$) a conjecture of Pilz on the symmetric difference  $A\Delta (2A)\Delta \cdots\Delta (nA)$ for finite sets $A\subseteq \mathbb{N}$ of positive integers. We show that this set always has cardinality at least $n$ for large $n$.
\end{abstract}

\maketitle

\setcounter{tocdepth}{1}
\tableofcontents

\section{Introduction}
We study a set operation on finite subsets of rings that arises naturally from parity type questions about the number of product representations and linear codes. Given a commutative ring $R$ and two finite sets $A,B\subseteq R$ and $x\in R$, we write $r_{A,B}(x)$ for the number of representations $x=ab$ with $a\in A$, $b\in B$. We define $A\ast B$ to be the set of those $x\in R$ such that $r_{A,B}(x)$ is odd.
We also use the notation $[n]$ to mean the set $\{1,2,\dots,n\}$. We write $a\cdot B=\{ab:b\in B\}$ and $\Delta$ for the symmetric difference of sets: $A\Delta B=(A\setminus B)\cup (B\setminus A)$.

The problem has its origins in coding theory. In \cite{p1992}, Pilz considered the minimal distance of a certain family of linear codes. The problem is equivalent to finding the minimum size of the set $A\ast [n]$ as $A$ ranges over all nonempty finite sets of natural numbers. Note that $A\ast [n]=A\Delta (2A)\Delta \cdots\Delta (nA)$. Pilz showed that when $n\le 6$, the set $A\ast [n]$ always has size at least $n$ \cite[Corollary 2]{p1992} and remarked that a computer search could not find any examples where $|A\ast [n]|<n$. In a later paper of Huang, Ke and Pilz \cite{hkp2010}, this was formally stated as a conjecture, which they called the Extended 1-2-3 Conjecture. We state this below and remark that the conjecture studied here is unrelated to the graph-theoretic 1-2-3 Conjecture.
\begin{conj}\label{pilz_conj}
For any $A\subseteq \bN$ finite and $n\in \bN$,
\[|A\:\Delta\: (2\cdot A)\:\Delta\cdots \Delta\: (n\cdot A)|\ge n.\]
\end{conj}
The special case $A=[k]$ (referred to as the 1-2-3 Conjecture) was resolved by Huang, Ke and Pilz \cite{hkp2010}, and independently by Szab\'o and the second author \cite{ps2011}. Our main result is a proof of Conjecture \ref{pilz_conj} for $n$ sufficiently large. 

\begin{thm}\label{pilz_large_n}
There is some effective constant $N$ such that for all $n\ge N$ and all $A\subseteq \bN$ finite and nonempty, we have $|A\:\Delta\: (2\cdot A)\:\Delta\cdots \Delta\: (n\cdot A)|\ge n$.

One may in particular take $N=2\cdot (3^{80}-321)$.
\end{thm}

The constant $2\cdot (3^{80}-321)<3\cdot 10^{38}$ is certainly not the best that our method can achieve, but the main goal of this paper is to prove that such a reasonably sized constant exists, while keeping the presentation easier to follow and avoiding further technical calculations. Better explicit bounds for the counts of rough numbers in certain intervals would lead to immediate improvements of our constant. We believe that further progress toward Conjecture~\ref{pilz_conj} will require new ideas beyond the estimates used in this paper.

\subsection{Notation.} We first introduce some notation that will be used throughout the paper. The standard notation $\ll$ is applied to positive quantities in the usual way, that is, $X \ll Y$ means that $X \leq cY$,
for some absolute constant $c > 0$. 
If the constant $c$ depends on a quantity
$t$, we write $X \ll_t Y$. We use the convention $\mathbb{N}=\{1,2,\dots\}$.

Let $A^{\ast k}$ denote $\underbrace{A \ast A \ast \dots  \ast A}_{k \text{ times}}$, noting that this is well-defined because $\ast$ is associative, as we will prove in Section \ref{prelim_section}. We will also sometimes write, given some sets $A_{m},A_{m+1},\dots,A_{n}$,
\[\bigdelta_{i=m}^{n}A_{i}=A_{m}\Delta \cdots \Delta A_{n},\]
and if $A_{i}$ are instead indexed by some other finite set $I$, then we similarly define
$\bigdelta_{i\in I} A_{i}$ to be the symmetric difference of these $A_{i}$, which is well-defined and the order does not matter because $\Delta$ is both associative and commutative. Indeed, $\bigdelta_{i\in I} A_{i}$ contains those elements that are contained in an odd number of the sets $A_i$.

\section*{Further results}
In this section we state some secondary results towards the resolution of Conjecture \ref{pilz_conj}. In particular, we resolve the conjecture for $n\le 8$ (Theorem \ref{g_le8}) and in the case where $|A|$ is small compared to $n$ (Theorem \ref{kn/logk_bound}). We will also consider the asymptotic size of $A\ast [n]$ as $n\rightarrow \infty$ with $A$ fixed.

\begin{thm}\label{g_le8}
For all $n\le 8$, and $A\subseteq \bN$ finite and nonempty, we have $|A\ast [n]|\ge n$.
\end{thm}

We can measure progress on Conjecture~\ref{pilz_conj} by defining
\[g(n)=\min_{A\subseteq \bN\:\text{finite, nonempty}}|A\ast [n]|\]
and finding a lower bound for $g(n)$. Pilz~\cite{p1992} proved the bound $\pi(n)+2\leq g(n)$ and the previous best known bound on $n$ is $g(n)\gg_{\lambda} n/(\log n)^{\lambda}$ for $\lambda>\lambda_{0}\approx 0.2223$, due to Szab\'o and the second author \cite[Theorem 2]{ps2011}. Note that the example $A=\{1\}$ means that $g(n)\le n$, so Pilz's conjecture is in fact equivalent to $g(n)= n$.

We can also define, for $B\subseteq \bN$ finite,
\[g(B)=\min_{A\subseteq \bN\:\text{finite, nonempty}}|A\ast B|.\]
It is not hard to see that the natural generalisation of Pilz's conjecture, that $g(B)=|B|$ for all finite $B\subseteq \bN$, does not hold. For example, if $B=\{1,2,4\}$ and $A=\{1,2\}$, then $A\ast B=\{1,8\}$. This means that a full proof of the conjecture must use some of the structure of the set $[n]$. A crucial property of $[n]$ that we will use is that there are many primes $p$ such that $p\in [n]$ and no larger multiple of $p$ is in $[n]$.

The main difficulty in establishing the conjecture is the large amount of cancellation that occurs. In many other problems in combinatorics, such as in Ramsey theory, we can pass from a large random set to a smaller set with more structure, but adding or removing even a single point from $A$ can dramatically affect the size of $A\ast [n]$. In fact Ke and Meyer proved in \cite[Theorem 5.1]{km2025} that for each $n$, there are arbitrarily large sets $A$ for which $|A\ast [n]|=n$. Let us show three examples where equality occurs: $A=\{a\}$, $A=\{a,2a\}$ and $n$ is even, and $A=[n]$, where $[n]*[n]=\{1^2,2^2,\dots,n^2\}$.  Following the proof of the aforementioned lower bound of  Szab\'o and the second author \cite{ps2011}, one finds that they actually obtain a lower bound for the number of integers which are represented {\it exactly once} as a product of $a\in A$ and $k\in [n]$. Such methods avoid having to deal with the cancellation, but they can never be sufficient to prove the full strength of Conjecture \ref{pilz_conj}. Note that if $B$ is a set of squarefree numbers, then it can be shown that there are at least $|B|$ products that have a unique representation as $ab$ ($a\in A,b\in B$). For $B=[n]$ this fails to hold, for instance, we have $[n]*[n]=\{1^2,2^2,\dots,n^2\}$, but most of the square numbers have multiple representations as $ab$ ($1\leq a,b\leq n$).

We recall an open problem from combinatorial geometry, first posed by Pak~\cite{Pak}, asking whether the area of the symmetric difference of an odd number of unit discs in $\mathbb{R}^2$ is always at least $\pi$. The problem remains open, see also \cite{Pin23} for further discussion. A key difference between this question and Conjecture~\ref{pilz_conj} is that the geometric problem is restricted to an odd number of sets, whereas no such parity restriction is present in our setting. We also note that the special case of Conjecture~\ref{pilz_conj} with $|A|$ odd would follow from the existence of an $n$-colouring of $\mathbb{N}$ such that $a,2a,\dots,na$ receive pairwise distinct colours for every $a\in\mathbb{N}$ \cite[Section 1.4]{CCP21}.

We can also consider what happens when we restrict the size of $A$. For $n,k\in \bN$, let
\[g(n,k)=\min_{\substack{A\subseteq \bN\\|A|=k}} |A\ast [n]|.\]
Ke and Meyer showed \cite{km2025} that $g(n,k)\ge n$ for $k\in \{1,2,3\}$ and all $n\in \bN$. Another main result of this paper will be to give a lower bound on $g(n,k)$ when $k$ is sufficiently small in terms of $n$.
\begin{thm}\label{kn/logk_bound}
Let $\eps>0$. Then for all $2\le k=|A|$ and $n\ge k^{2+\eps}$, we have
\[|A\ast [n]|\gg_{\eps} \frac{kn}{\log k}.\]
\end{thm}

Henceforth, we will write $a_{1},\dots,a_{k}$ for the elements of a finite set $A\subseteq \bN$, where we may suppose without loss of generality that $a_{1}<\cdots<a_{k}$. We shall use the following inclusion-exclusion formula for the size of $A*[n]$:
\begin{lemma}\label{exact_folmula}
For any finite set $A\subseteq \bN$ and $n\ge 1$, we have
\begin{equation}\label{exact_formula_eq}|A\ast [n]|=\sum_{r=1}^{k}\sum_{1\le i_{1}<\cdots<i_{r}\le k} (-2)^{r-1}\left\lfloor\frac{a_{i_{1}}n}{[a_{i_{1}},\dots,a_{i_{r}}]} \right\rfloor.\end{equation}
\begin{proof}
This is exactly \cite[(2.1)]{km2025}.
\end{proof}
\end{lemma}
\begin{lemma}\label{hA_limit_identity}
For $A\subseteq \bN$ finite, the limit
\[\lim_{n\rightarrow\infty} \frac{|A\ast [n]|}{n}\]
exists, and is equal to
\begin{equation}\label{hA_exact_formula}h(A):=\sum_{r=1}^{k}\sum_{1\le i_{1}<\cdots<i_{r}\le k} (-2)^{r-1}\frac{a_{i_{1}}}{[a_{i_{1}},\dots,a_{i_{r}}]}.\end{equation}
\end{lemma}
\begin{proof}
This follows easily from Lemma \ref{exact_folmula}.
\end{proof}

Define
\[h(k)=\inf_{\substack{A\subseteq \bN\\ |A|=k}} h(A).\]
Theorem \ref{kn/logk_bound} implies a lower bound for $h(k)$.
\begin{cor}\label{hk_bound_k_over_logk}
For  $k\ne 1$, we have
\[h(k)\gg \frac{k}{\log k}.\]
\begin{proof}
This follows immediately from Theorem \ref{kn/logk_bound} and Lemma \ref{hA_limit_identity}.
\end{proof}
\end{cor}
We also have an explicit uniform lower bound for $h(A)$.
\begin{thm}\label{hA_uniform_lower_bound}
For any finite nonempty $A\subseteq \bN$, either
\[h(A)\ge \frac{8}{7}\]
or $A$ is of the form $\{a\}$ or $\{a,2a\}$ for some $a\in \bN$.
\end{thm}
\begin{remark} We expect that the same theorem should hold with $8/7$ replaced with $4/3$. Our method could achieve this up to checking finitely many cases, but the number of cases to check is too large.
\end{remark}

\subsection{Organisation and proof strategy.}

In Section \ref{prelim_section}, we prove some basic facts about the binary operation $\ast$ that will be useful later.

Then in Section \ref{small_n_section}, we prove Conjecture \ref{pilz_conj} for $n\le 8$, using a method of dividing the set $A$ into slices according to the $p$-adic valuation for a certain prime $p$.

In Section \ref{hA_section}, we prove Theorem \ref{kn/logk_bound}. We prove some results that allow us to split up $A$ into smaller sets $B$ and $C$ such that $|A\ast [n]|$ is at least $|B\ast [n]|+|C\ast [n]|$ minus a small loss. We also show a lower bound for $|A\ast [n]|$ that applies when all of the elements of $A$ are sufficiently smooth and $n$ is sufficiently large in terms of $|A|$. Combining these results, we are able to prove Theorem \ref{kn/logk_bound}.

Then in Section \ref{hA_bound_section}, we prove some explicit lower bounds for $h(k)$, including Theorem \ref{hA_uniform_lower_bound}. We also compute the exact values of $h(3)$ and $h(4)$. These proofs involve using the results of the previous section to reduce the proof to checking finitely many cases, and then using a computer search.

Finally, in Section \ref{large_n_section}, we prove Theorem \ref{pilz_large_n}. The basic idea is the same $p$-adic slicing argument from Section \ref{small_n_section}, but is more involved. Most of the work goes into proving a lower bound for the size of sets of the form
\[\left(A_{1}\ast ([n]\setminus \mathcal{Q})\right)\Delta \cdots \Delta \left(A_{r}\ast ([n]\setminus \mathcal{Q})^{\ast r}\right),\]
where the $A_{i}$ are small in terms of $n$ and $\mathcal{Q}$ is a small set of primes. This requires results from Sections \ref{hA_section} and \ref{hA_bound_section}.

\section{The algebraic framework}\label{prelim_section}
In this section, we prove some basic results about the binary operation $\ast$ that will be useful later. Given a set $S$, we write $\mathcal{F}(S)$ to be the set of all finite subsets of $S$. Then $\ast$ is a binary operation on $\mathcal{F}(\bN)$ and it turns out that, together with the symmetric difference operator $\Delta$, it turns this set into a ring.
\begin{lemma}\label{ring}
The set $\mathcal{F}(\bN)$ with the binary operations $\Delta$ and $\ast$ form a commutative ring of characteristic $2$, with additive identity $\emptyset$ and multiplicative identity $\{1\}$. That is, $\ast$ and $\Delta$ are both commutative and associative and for every $A,B,C\subseteq \bN$ finite,
\[A\ast (B\Delta C)=(A\ast B)\Delta (A\ast C)\quad \text{(distributivity),}\]
\[A\Delta \emptyset=A\quad \text{(additive identity)},\]
\[A\ast \{1\}=A\quad \text{(multiplicative identity)},\]
and
\[A\Delta A=\emptyset \quad \text{(characteristic $2$)}.\]
\begin{proof}
These are all fairly immediate except for distributivity and the associativity of $\ast$. In both proofs, we will use the fact that for all $A,B\in \mathcal{F}(\bN)$ and $x\in \bN$
\begin{equation}\label{rAB_parity}r_{A,B}(x)\equiv 1_{A\ast B}(x) \: \text{(mod $2$)},\end{equation}
which follows from the definition of $\ast$.

Let $A,B,C\in \mathcal{F}(\bN)$. For any $x\in \bN$,
\[r_{A,B\Delta C}(x)=\sum_{ab=x} 1_{A}(a)1_{B\Delta C}(b).\]
For any $b\in \bN$, we have
\[1_{B\Delta C}(b)\equiv 1_{B}(b)+1_{C}(b)\: \text{(mod $2$)},\]
so
\[r_{A,B\Delta C}(x)\equiv r_{A,B}(x)+r_{A,C}(x)\: \text{(mod $2$)}.\]
It follows that $A\ast (B\Delta C)=(A\ast B)\Delta (A\ast C)$.
We also have
\[r_{A,B\ast C}(x)=\sum_{ab=x} 1_{A}(a)1_{B\ast C}(b)\equiv \sum_{ab=x} 1_{A}(a)r_{B,C}(b)\: \text{(mod $2$)}.\]
Then
\[\sum_{ab=x} 1_{A}(a)r_{B,C}(b)=\sum_{abc=x} 1_{A}(a)1_{B}(b)1_{C}(c).\]
By a similar argument, we also have
\[r_{A\ast B, C}(x)\equiv \sum_{abc=x} 1_{A}(a)1_{B}(b)1_{C}(c)\: \text{(mod $2$)},\]
which implies that $A\ast (B\ast C)=(A\ast B)\ast C$.
\end{proof}
\end{lemma}

\begin{lemma}\label{ring_iso}
Let $p_{i}$ be the $i$th prime number and $\bF_{2}$ the finite field of order $2$. The ring homomorphism
\[\varphi:\bF_{2}[X_{1},X_{2},\dots]\rightarrow \mathcal{F}(\bN)\]
given by
\[\varphi(X_{i})=\{p_{i}\}\]
is an isomorphism, and for every $f\in \bF_{2}[X_{1},X_{2},\dots]$, the cardinality of $\varphi(f)$ is equal to the number of nonzero coefficients of $f$.
\begin{proof}
Given $\bk=(k_{1},k_{2},\dots)$ with $k_{i}\in \bZ_{\ge 0}$ for all $i$ and $k_{i}=0$ for all $i>n$, say, we write $X^{\bk}$ to denote the monomial $X_{1}^{k_{1}}\cdots X_{n}^{k_{n}}$ and we write $P^{\bk}=p_{1}^{k_{1}}\cdots p_{n}^{k_{n}}$. We have that
\[\varphi\left(X^{\bk}\right)=\{P^{\bk}\}.\]
Any $f\in \bF_{2}[X_{1},X_{2},\dots]$ can be written as
\[f=\sum_{i=1}^{m} X^{\bk_{i}}\]
for some distinct $\bk_{1},\dots,\bk_{m}$. Then we can calculate
\[\varphi(f)=\{P^{\bk_{1}},\dots,P^{\bk_{m}}\}.\]
By uniqueness of prime factorisation, $\varphi$ is injective and $\varphi(f)$ has cardinality $m$. It is also surjective since every $x\in \bN$ can be written as $P^{\bk}$ for some $\bk$.
\end{proof}
\end{lemma}
In light of this isomorphism, for a polynomial $f\in \bF_{2}[X_{1},X_{2},\dots]$, we define $g(f)$ to be the minimum number of nonzero coefficients of $f\cdot h$ over all $h\in \bF_{2}[X_{1},X_{2},\dots]\setminus \{0\}$. If $\varphi$ is the isomorphism from Lemma \ref{ring_iso} then $g(f)=g(\varphi(f))$ for all $f$.
\begin{remark}\label{variable_permute}We note the following fact which will be useful later. If $\alpha$ is the automorphism of $\bF_{2}[X_{1},X_{2},\dots]$ which maps $X_{i}$ to $X_{\sigma(i)}$ for some bijection $\sigma:\bN\rightarrow \bN$, then $g(f)=g(\alpha(f))$ for all $f$.\end{remark}

Given $s\in \bN$, we say that an integer $n$ is $s$-smooth if for any prime $p$, if $p\mid n$ then $p\le s$. We say that it is $s$-rough if for every prime $p$, if $p\mid n$ then $p\ge s+1$. Every $n\in \bN$ decomposes uniquely as a product $n=ab$ where $a$ is $s$-smooth and $b$ is $s$-rough. Note that $1$ is $s$-smooth and $s$-rough for any $s$. 

\begin{lemma}\label{rough_decomp}
Let $A\subseteq \bN$ be finite and $n\in\bN$. Then there are unique $1\le b_{1}<\cdots<b_{r}$ which are $n$-rough, and sets $C_{i}$ for $1\le i\le r$ such that every $c\in C_{i}$ is $n$-smooth and $A$ is the disjoint union
\[A=\bigcup_{i=1}^{r} b_{i}\cdot C_{i}.\]
Furthermore, we have
\begin{equation}\label{rough_decomp_star_n}A\ast [n]=\bigcup_{i=1}^{r} b_{i}\cdot (C_{i}\ast [n]),\end{equation}
and
\begin{equation}\label{rough_decomp_star_size}|A\ast [n]|=\sum_{i=1}^{r} |C_{i}\ast [n]|.\end{equation}
\begin{proof}
For each $a\in A$, let $a=s(a)r(a)$ where $s(a)$ is $n$-smooth and $r(a)$ is $n$-rough. These $s(a)$, $r(a)$ are uniquely determined by $a$. Let $b_{1}<b_{2}<\cdots<b_{r}$ be the distinct values that are taken by $r(a)$ for $a\in A$ and let $C_{i}=\{s(a):a\in A,\: r(a)=b_{i}\}$. Then it is easily checked that the sets $b_{i}\cdot C_{i}$ are pairwise disjoint and their union is $A$.

If $a,a'\in A$ and $m,m'\in [n]$ are such that $am=a'm'$ then we claim that $a,a'\in b_{i}\cdot C_{i}$ for some $i$. From this, \eqref{rough_decomp_star_n} and \eqref{rough_decomp_star_size} will follow. To prove the claim, we note that $am=s(a)mr(a)=s(a')m'r(a')=a'm'$. Every $m,m'\in [n]$ is $n$-smooth, so by uniqueness of the smooth-rough decomposition, we must have $s(a)m=s(a')m'$ and $r(a)=r(a')=b_{i}$ for some $i$. Then by definition of $C_{i}$, we find that $s(a),s(a')\in C_{i}$.
\end{proof}
\end{lemma}

\section{Verifying the conjecture for $n\le 8$}\label{small_n_section}
Our aim in this section is to prove Theorem \ref{g_le8}, that is, that $g(n)=n$ for all $n\le 8$. By Lemma \ref{rough_decomp}, to check that $g(n)\ge n$, it suffices to check those $A$ with every $a\in A$ $n$-smooth. We will therefore assume without loss of generality in the proofs in this section that $A$ consists only of $n$-smooth numbers. The cases $n\le 6$ are already known from the work of Pilz \cite{p1992}, so we just need to prove that $g(7)=7$ and $g(8)=8$.

For a prime $p$, we write $\nu_{p}$ for the $p$-adic valuation. We also write
\[\nu_{p}(A)=\max\{\nu_{p}(a): a\in A\},\]
and in this section, for $i\in \bZ$,
\[A_{p}^{(i)}=\{a\in A:\nu_{p}(a)=i\}.\]
Note that for any fixed $p$, $A$ is the disjoint union $\bigcup_{i\in \bZ}A_{p}^{(i)}$.
\begin{lemma}\label{star_slice_formula}
Let $S\subseteq \bN$ and $p$ a prime such that $\nu_{p}(x)\le 1$ for all $x\in S$. Let $S_{0}=\{x\in S:p\nmid x\}$, $S_{1}=\{x\in S:p\mid x\}$. Then
\[|A\ast S|=\sum_{i=0}^{\nu_{p}(A)+1}\left|\left(A_{p}^{(i-1)}\ast S_{1}\right)\Delta \left(A_{p}^{(i)}\ast S_{0}\right)\right|.\]
\begin{proof}
Let
\[B_{p}^{(i)}=\{b\in A\ast S:\nu_{p}(b)=i\}.\]
Then it is not hard to show that
\[B_{p}^{(i)}=\left( A_{p}^{(i-1)}\ast S_{1}\right)\Delta \left(A_{p}^{(i)}\ast S_{0}\right).\]
The result follows.
\end{proof}
\end{lemma}
\begin{lemma}\label{prime_g_bound}
Let $p$ be a prime. Then $g(p)\ge g(p-1)+1$.
\begin{proof}
Without loss of generality, we may assume that $A_p^{(0)}\ne\emptyset$. 
The claimed bound immediately follows from Lemma \ref{star_slice_formula} with $S=[p]$ upon noting that
\[\left|A_{p}^{(\nu_{p}(A))}\ast \{p\}\right|=\left| A_{p}^{(\nu_{p}(A))}\right|\ge 1\]
and
\[\left|A_{p}^{(0)}\ast [p-1]\right|\ge g(p-1).\]
\end{proof}
\end{lemma}

The following Lemma is \cite[Corollary 2]{p1992}.
\begin{lemma}[Pilz]\label{g_le6}
For $n\le 6$, $g(n)=n$.
\end{lemma}
The following is a generalisation of \cite[Lemma 4]{p1992}, which states that when $n$ is even, then so is $|A\ast [n]|$ for any finite $A\subseteq\bN$.
\begin{lemma}\label{parity_lemma}
If $A,B\subseteq \bN$ are finite and nonempty, then
\[|A\ast B|\equiv |A|\cdot |B|\: \text{(mod $2$)}\]
and
\[|A\ast B|\ge 1.\]
\begin{proof}
Recall the identity \eqref{rAB_parity}
\[r_{A,B}(x)\equiv 1_{A\ast B}(x)\: \text{(mod $2$)},\]
for all $x\in \bN$. Then the first part follows from the facts that $|A\ast B|=\sum_{x\in \bN} 1_{A\ast B}(x)$ and $|A|\cdot |B|=\sum_{x\in \bN} r_{A,B}(x)$.

For the second part, note that if $x=\max A \cdot \max B$, then $r_{A,B}(x)=1$, so $x\in A\ast B$.
\end{proof}
\end{lemma}

\begin{lemma}\label{g8}
We have $g(8)=8$.
\begin{proof}
Applying Lemma \ref{star_slice_formula} with $p=3$, we have
\begin{equation}\label{g8_slice}|A\ast [8]|\ge \left|A_{3}^{(\nu_{3}(A))}\ast \{3,6\}\right|+\left|A_{3}^{(0)}\ast \{1,2,4,5,7,8\}\right|\ge 2+g(\{1,2,4,5,7,8\}).\end{equation}
By Lemma \ref{parity_lemma}, $g(\{1,2,4,5,7,8\})$ is even, so it suffices to show that
\[g(\{1,2,4,5,7,8\})\ge 5.\]
By Lemma \ref{star_slice_formula} with $p=7$,
\[|A\ast \{1,2,4,5,7,8\}|\ge \left|A_{7}^{(\nu_{7}(A))}\right|+\left|A_{7}^{(0)}\ast \{1,2,4,5,8\}\right|\]
with equality if and only if, for all $1\le i\le \nu_{7}(A)$,
\[7\cdot A_{7}^{(i-1)}=A_{7}^{(i)}\ast \{1,2,4,5,8\}.\]
Also, by Lemmas \ref{star_slice_formula} and \ref{parity_lemma}
\[g(\{1,2,4,5,8\})\ge g(\{1,2,4,8\})+1\ge 3.\]
Hence, either
\begin{equation}\label{g8_eq1}\left|A\ast \{1,2,4,5,7,8\}\right|\ge g(\{1,2,4,5,8\})+2\ge 5\end{equation}
or $|A_{7}^{(\nu_{7}(A))}|=1$ and
\begin{equation}\label{g8_eq2}\left|A_{7}^{(0)}\ast \{1,2,4,5,8\}\right|=\left|\{1,2,4,5,8\}^{\ast \nu_{7}(A)}\right|.\end{equation}
We claim that $\{1,2,4,5,8\}^{\ast m}$ has cardinality at least $5$ for all $m\ge 1$. By Lemma \ref{ring_iso}, this is equivalent to showing that $(1+X+X^{2}+X^{3}+Y)^{m}$ has at least $5$ nonzero coefficients in $\bF_{2}[X,Y]$. It suffices to show that $(1+X+X^{2}+X^{3})^{m}$ has at least $4$ nonzero coefficients. This follows from \cite[Theorem 1.3]{np2026}. Translated into polynomial language, this theorem says that for any $r\in \bN$ and $a_{1},\dots,a_{k}\in \bN$ (not necessarily distinct), the polynomial $\prod_{i=1}^{k} (1+X^{a_{i}}+\cdots+X^{(r-1)a_{i}})$ has at least $r$ nonzero coefficients.

It follows that
\[\left|\{1,2,4,5,8\}^{\ast \nu_{7}(A)}\right|\ge 5,\]
and then, by \eqref{g8_slice}
\[|A\ast [8]|\ge 7.\]
But by Lemma \ref{parity_lemma}, $|A\ast [8]|$ is even, so the result follows.
\end{proof}
\end{lemma}

\begin{proof}[Proof of Theorem \ref{g_le8}]
We combine Lemmas \ref{g_le6} and \ref{g8}, and use Lemma \ref{prime_g_bound} to handle the case $n=7$.
\end{proof}

\section{The case of small $A$}\label{hA_section}
The aim of this section is to prove Theorem \ref{kn/logk_bound}. Our argument will essentially be by induction on the size of $A$, and the following lemma will be instrumental in the inductive step.
\begin{lemma}\label{partition_bound}
Suppose that $A=B\cup C$ for disjoint $B,C$ and let $n\in \bN$. Then for all $S\subseteq [n]$,
\[|B\ast S|+|C\ast S|-2\sum_{b\in B}\sum_{c\in C}\left\lfloor\frac{(b,c)n}{\max\{b,c\}}\right\rfloor\le |A\ast S|\le |B\ast S|+|C\ast S|,\]
so in particular,
\[h(B)+h(C)-2\sum_{b\in B}\sum_{c\in C}\frac{(b,c)}{\max\{b,c\}}\le h(A)\le h(B)+h(C).\]
\begin{remark}
The assumption that $B$, $C$ are disjoint can be removed with some extra effort, but we will not need this.
\end{remark}
\begin{proof}
By the disjointness of $B,C$ we have
\[A=B\cup C=B\Delta C,\]
so by distributivity (Lemma \ref{ring}),
\[A\ast S=(B\ast S)\Delta (C\ast S)\]
and
\[|B\ast S|+|C\ast S|\ge |A\ast S|\ge |B\ast S|+|C\ast S|-2\left|(B\ast S)\cap (C\ast S)\right|.\]
It suffices to show that
\[\left|(B\ast S)\cap (C\ast S)\right|\le \sum_{b\in B}\sum_{c\in C}\left\lfloor\frac{(b,c)n}{\max\{b,c\}}\right\rfloor.\]
Since $S\subseteq [n]$, we clearly have,
\[B\ast S\subseteq \bigcup_{b\in B}b\cdot [n],\]
and similarly for $C$. It follows that
\[\left|(B\ast S)\cap (C\ast S)\right|\le \sum_{b\in B}\sum_{c\in C}\left|(b\cdot [n])\cap (c\cdot [n])\right|.\]
We can then check that
\[\left|(b\cdot [n])\cap (c\cdot [n])\right|=\left\lfloor\frac{(b,c)n}{\max\{b,c\}}\right\rfloor.\]
For the second part, we take $S=[n]$ and recall the definition of $h$.
\end{proof}
\end{lemma}

In the proof of Theorem \ref{kn/logk_bound}, we will use Lemma \ref{partition_bound} in the form of the following corollary.
\begin{cor}\label{large_prime_split}
Suppose that $|A|=k\ge 2$, $\gcd(A)=1$ and suppose that $p$ is a prime dividing $a$ for some $a\in A$. Then there are $B$, $C$ nonempty and disjoint sets such that $A=B\cup C$, $m=|B|$, and for all $n\in \bN$ and $S\subseteq [n]$,
\[|A\ast S|\ge |B\ast S|+|C\ast S|-n\cdot \frac{2m(k-m)}{p},\]
so that
\[h(A)\ge h(B)+h(C)-\frac{2m(k-m)}{p}.\]
\begin{proof}
Let
\[B=\{a\in A: p\mid a\}\]
and $C=A\setminus B$. Then $B$ is nonempty by assumption and $C$ is nonempty because otherwise $p\mid \gcd(A)$. For every $b\in B$ and $c\in C$, we have $p\mid b$ but $p\nmid c$, so
\[\frac{(b,c)}{b}\le \frac{1}{p}.\]
The result then follows immediately from Lemma \ref{partition_bound}.
\end{proof}
\end{cor}

Corollary \ref{large_prime_split} allows us to deal with those $A$ where a large prime divides some $a\in A$. The remaining $A$ are those with all $a\in A$ smooth. We deal with these in the following lemma.
\begin{lemma}\label{smooth_bound_A_star_n}
Suppose that $A\subseteq \bN$ with $|A|=k$ is such that every $a\in A$ is $s$-smooth. If $2\le s\le n/2$, then we have
\begin{equation}\label{A_star_n_smooth_bound}\left|A\ast [n]\right|\gg \frac{kn}{\log s}.\end{equation}
Furthermore, if $n\ge 41$, $7\le s\le n/2$ and $T\subseteq [n]$ with $\alpha=|T|/n$ then we have the explicit bound
\begin{equation}\label{A_star_n_smooth_bound_explicit}\left|A\ast ([n]\setminus T)\right|\ge kn\left(\frac{1}{10\log s}-\alpha\right).\end{equation}
\begin{proof}
Let $T\subseteq [n]$ and $2\le s\le n/2$ and let $S(n)$ be the set of all $s$-rough numbers in the range $(n/2,n]$. These numbers are all coprime to $a_{1}\cdots a_{k}$.

Suppose $x\in S(n)\setminus T$ and $1\le i\le k$. We claim that $r_{A,[n]\setminus T}(a_{i}x)=1$ and hence $a_{i}x\in A\ast ([n]\setminus T)$. To see this, suppose that $y\in [n]$ and $1\le j\le k$ with $a_{i}x=a_{j}y$. Then $x\mid a_{j}y$, and since $x$ and $a_{j}$ are coprime, we must have $x\mid y$. But since $n/2<x\le y\le n$, this means that $y=x$ and hence also $a_{i}=a_{j}$. It follows that
\[\left|A\ast ([n]\setminus T)\right|\ge k|S(n)\setminus T|.\]
We first suppose that $n\ge 41$ and $s\ge 7$ and show \eqref{A_star_n_smooth_bound_explicit}. For $x,y\ge 1$, we define $\Phi(x,y)$ to be the number of $m\le x$ which are $y$-rough. Note that $|S(n)|=\Phi(n,s)-\Phi(n/2,s)$, so for $x\ge 10$ and $2\le y\le x^{1/2}$, \cite[Theorem 1]{fp2024} gives the upper bound
\[\Phi(x,y)<\frac{0.6 x}{\log y}\]
and \cite[Theorem 2.3]{f2024} tells us that for all $7\le y\le x^{2/3}$, we have
\[\Phi(x,y)>\frac{0.4x}{\log y}.\]
So for all $n\ge 10$ and $7\le s\le (n/2)^{1/2}$, we have
\[|S(n)|\ge \frac{0.4 n}{\log s}-\frac{0.3 n}{\log s},\]
which implies \eqref{A_star_n_smooth_bound_explicit} when $s\le (n/2)^{1/2}$ (note that $n\ge 10$ is satisfied automatically because otherwise $n/2<7$). When $(n/2)^{1/2}<s\le n/2$, we have $\log s>(\log n)/2$. Also, $S(n)$ contains all of the primes in the range $(n/2,n]$, so $|S(n)|\ge \pi(n)-\pi(n/2)$, which by \cite[Corollary 3]{rs1962} is at least $3 n/(10 \log n)\ge 3 n/(20 \log s)$ when $n\ge 41$, which is sufficient to show \eqref{A_star_n_smooth_bound_explicit}.

We then have \eqref{A_star_n_smooth_bound} in the case where $n\ge 41$ and $s\ge 7$. In the remaining cases, we have that
\[\Phi(x,6)=x\prod_{p\le 6}\left(1-\frac{1}{p}\right)+O(1),\]
which deals with the case $s<7$ and $n$ sufficiently large. For $n\ll 1$, by Bertrand's postulate, $|S(n)|\ge 1$, so $|A\ast [n]|\ge 1$. Then \eqref{A_star_n_smooth_bound} follows in general.
\end{proof}
\end{lemma}

\begin{proof}[Proof of Theorem \ref{kn/logk_bound}]
Let $\eps>0$, $k\ge 2$, $H\ge 1/2$, $n\ge H^{1+\eps} k^{2+\eps}$, and suppose that for all $1\le i\le k-1$, and all $A\subseteq \bN$ with $|A|=i$, we have
\[|A\ast [n]|\ge \frac{in}{H\log (i+1)}.\]
Suppose that $A=\{a_{1},\dots,a_{k}\}$ for some positive integers $a_{1}<\cdots<a_{k}$ and suppose without loss of generality that $\gcd(A)=1$. For each $B,C\subseteq A$ disjoint and nonempty sets with $A=B\cup C$, we have
\[|B\ast [n]|+|C\ast [n]|\ge \frac{in}{H\log (i+1)}+\frac{(k-i)n}{H\log (k-i+1)}.\]
Since the function $x/\log(x+1)$ is concave, the right-hand side is minimal for $i\in\{1,k-1\}$. 
We can also check that, for all $k\ge 2$,
\[\frac{1}{\log 2}+\frac{k-1}{\log k}-\frac{k}{\log (k+1)}>1,\]
and it follows that for all such decompositions $A=B\cup C$,
\[|B\ast [n]|+|C\ast [n]|> n\left(\frac{k}{H\log (k+1)}+\frac{1}{H}\right).\]
Then Corollary \ref{large_prime_split} tells us that
\[|A\ast [n]|\ge \frac{kn}{H\log (k+1)},\]
provided that the largest prime $p$ dividing $a_{1}\cdots a_{k}$ satisfies
\[\frac{k^{2}}{2p}\le H^{-1}.\]
This inequality holds if and only if
\[p\ge \frac{H k^{2}}{2}.\]
So either
\[|A\ast [n]|\ge \frac{kn}{H\log (k+1)},\]
or $p<H k^{2}/2$, in which case
\[\log p\ll \log H+\log k.\]
Recall that $n\ge H^{1+\eps}k^{2+\eps}$. Without loss of generality, it can be assumed that $\varepsilon$ is sufficiently small, then we have
\[\frac{H k^{2}}{2}<n^{1-\eps/3},\]
so by Lemma \ref{smooth_bound_A_star_n}, there exists a constant $C_{\eps}>0$ independent of $k$ and $H$ such that
\[|A\ast [n]|\ge \min \left\{\frac{kn}{H\log (k+1)},\frac{C_{\eps}kn}{\log k+\log H}\right\}.\]

Now, for $k\ge 2$, let
\[H_{k}=\inf\left\{H>1/2:\sup_{n\ge H^{1+\eps}k^{2+\eps}}\max_{\substack{1\le i\le k-1\\ |A|=i}} \frac{in}{|A\ast [n]|\log (i+1)}\le H\right\},\]
or $H_{k}=\infty$ if the set is empty. From the above argument, we have, for all $k\ge 3$, the recursive formula
\[H_{k}\le  \max \left\{H_{k-1},\frac{\log k+\log H_{k-1}}{C_{\eps}\log (k+1)}\right\}.\]
This in particular tells us that $H_{k}$ is finite for all $k$. Let $k_{0}\ge 3$ be such that $\log (k_{0}+1)\ge 1/C_{\eps}$. Then for $k\ge k_{0}$,
\[\frac{\log k+\log H_{k}}{C_{\eps}\log (k+1)}\le \frac{1}{C_{\eps}}+\log H_{k}.\]
Because $\log$ grows very slowly, there exists $\tilde{H}$ such that $1/C_{\eps}+\log H< H$, for all $H\ge \tilde{H}$. Then for all $k\ge k_{0}$,
\[H_{k+1}\le \max\{H_{k},\tilde{H}\}.\]
Hence there is some $H$, depending only on $\eps$, such that $H_{k}\le H$ for all $k$. Now, for $k\ge k_{1}:=\max(k_0,H^{1+1/\eps})$, we have $H^{1+\eps}k^{2+\eps}\le k^{2+2\eps}$. So recalling the definition of $H_{k}$, for $k\ge k_{1}$ and $n\ge k^{2+2\eps}$, we have
\[|A\ast [n]|\ge \frac{kn}{H\log (k+1)}.\]
We also recall that $H_{k}$ is finite for all $n$, which deals with the cases $k<k_{1}$. The result is easily deduced.
\end{proof}

\section{Explicit lower bounds on $h(k)$}\label{hA_bound_section}
We aim to find explicit lower bounds for $h(k)$. When $k$ is large, the methods of the previous section suffice, but when $k$ is small, more ideas are required. Lemma \ref{ratio_split} below allows us to bound $h(A)$ when the ratio of the largest to smallest element of $A$ is sufficiently large, and Corollary \ref{large_min_bound} will allow us to bound $h(A)$ when the smallest element is large. Together, these lemmas leave only finitely many $A$, and these can in theory be checked with a computer.

For $0\le \alpha<\beta$, let
\[[n]_{\alpha,\beta}=\bZ\cap (n\alpha,n\beta],\]
and
\[h_{\alpha,\beta}(A)=\lim_{n\rightarrow\infty}\frac{|A\ast [n]_{\alpha,\beta}|}{n},\]
where the limit exists by a standard application of the inclusion-exclusion principle. Clearly,
\[h_{0,\beta}(A)=\beta h(A).\]
\begin{lemma}\label{shift_limit}
For any finite set $A\subseteq \bN$ and $0\le \alpha<\beta$, we have
\[h_{\alpha,\beta}(A)\ge (\beta-\alpha)h(A).\]
\begin{proof}
We have that $[n]_{0,\beta}$ is the disjoint union of $[n]_{\alpha,\beta}$ and $[n]_{0,\alpha}$ and hence, by Lemma \ref{ring},
\[A\ast [n]_{0,\beta}=(A\ast [n]_{\alpha,\beta})\Delta (A\ast [n]_{0,\alpha}).\]
It follows that
\begin{align*}h(A)\beta n +o(n)=|A\ast [n]_{0,\beta}|&\le |A\ast [n]_{\alpha,\beta}|+|A\ast [n]_{0,\alpha}|\\&= |A\ast [n]_{\alpha,\beta}|+h(A)\alpha n+o(n).\end{align*}
The result then follows easily.
\end{proof}
\end{lemma}

\begin{lemma}\label{ratio_split}
Let $k\in\bN$, $A=\{a_{1},\dots, a_{k}\}$, where $1\le a_{1}<\cdots<a_{k}$. Then for $1\le i<k$,
\begin{equation}\label{highratio_bound}h(A)\ge h(\{a_{1},\dots ,a_{i}\})+\left(1-3\frac{a_{i}}{a_{i+1}}\right)h(\{a_{i+1},\dots ,a_{k}\}).\end{equation}
Hence, if $r\ge 3$ and $a_{k}/a_{1}\ge r^{k-1}$, then
\[h(A)\ge \min_{1\le i\le k-1} \left(h(i)+(1-3/r)h(k-i)\right).\]
\begin{proof}
Let $B=\{a_{1},\dots,a_{i}\}$, $C=\{a_{i+1},\dots,a_{k}\}$ and $\alpha=a_{i}/a_{i+1}$. If $1\le j\le i$ and $i+1\le \ell\le k$ and $a_{j}m_{1}=a_{\ell}m_{2}$ for some $m_{1},m_{2}\in [n]$, then
\[n\ge m_{1}=\frac{a_{\ell}}{a_{j}}m_{2}\ge \frac{m_{2}}{\alpha},\]
so
\[m_{2}\le \alpha n.\]
Hence, $C\ast [n]_{\alpha,1}$ is disjoint from $B\ast [n]$. It follows that $A\ast [n]=(B\ast [n])\Delta (C\ast [n])$ contains the disjoint sets
\[(B\ast [n])\setminus (C\ast [n]_{0,\alpha})\]
and
\[(C\ast [n]_{\alpha,1})\setminus (C\ast [n]_{0,\alpha}).\]
By Lemma \ref{shift_limit}, $h_{\alpha,1}(C)\ge (1-\alpha)h(C)$, so these sets have cardinality at least $(h(B)-\alpha h(C))n+o(n)$ and $(1-2\alpha)h(C)n+o(n)$ respectively. Then
\[h(A)\ge h(B)+(1-3\alpha)h(C),\]
which gives the first part.

For the second part, simply note that if $a_{k}/a_{1}\ge r^{k-1}$, then $a_{i}/a_{i+1}\le 1/r$ for some $i$ and then use \eqref{highratio_bound}.
\end{proof}
\end{lemma}

\begin{lemma}\label{large_min_split}
Suppose that $\gcd(A)=1$ and $a_{1}\ge R\ge 2$. Then there exist sets $B,C$ with the following properties.
\begin{enumerate}
\item The sets $B$, $C$ are nonempty.
\item The set $A$ is the disjoint union $B\cup C$.
\item Let $m=\gcd(B)$. Then for all $c\in C$, we have
\[\frac{m}{(m,c)}\ge R^{1/(k-1)}.\]
\end{enumerate}
\begin{proof}
Assume for a contradiction that no such $B$, $C$ exist. We will inductively construct a sequence $B_{i}$, $C_{i}$ that will lead to a contradiction. Let $B_{0}=\{a\in A:a_{1}\mid a\}$ and $C_{0}=A\setminus B_{0}$. Then $B_{0}$ is nonempty because it contains $a_{1}$ and $C_{0}$ is nonempty because otherwise $\gcd(A)=a_{1}$. Let $m_{0}=\gcd(B_{0})=a_{1}$ and note that $m_{0}\ge R$.

Now suppose as an inductive hypothesis that we have defined $B_{i}$, $C_{i}$ nonempty such that $|B_{i}|\ge i+1$, $A$ is the disjoint union of $B_{i}$ and $C_{i}$ and $m_{i}=\gcd(B_{i})\ge R^{1-i/(k-1)}$.
By the assumption that the lemma is false, there must be some $c'\in C_{i}$ such that
\begin{equation}\label{mi_over_gcd_upper_bound}\frac{m_{i}}{(m_{i},c')}<R^{1/(k-1)}.\end{equation}
Define $C_{i+1}=C_{i}\setminus \{c'\}$ and $B_{i+1}\cup \{c'\}$. Then
\[m_{i+1}=\gcd(B_{i+1})=(m_{i},c')>m_{i}R^{-1/(k-1)}\ge R^{1-(i+1)/(k-1)}.\]
We also clearly have $|B_{i+1}|=|B_{i}|+1\ge i+2$ and $A$ is the disjoint union of $B_{i+1}$ and $C_{i+1}$. To complete the induction step, it just remains to show that $C_{i+1}$ is nonempty. If $C_{i+1}=\emptyset$ then $C_{i}=\{c'\}$ and we must have $(m_{i},c')=\gcd(A)=1$. So $m_{i}<R^{1/(k-1)}$ by \eqref{mi_over_gcd_upper_bound}. But the induction hypothesis tells us that $m_{i}\ge R^{1-i/(k-1)}$. Note that $i+1\le |B_{i}|=|A|-1=k-1$, so $i\le k-2$ and $m_{i}\ge R^{1/(k-1)}$, which is a contradiction. Hence $C_{i+1}$ is nonempty.

By induction, we have such a pair $B_{i}$, $C_{i}$ for every $i\ge 0$, but this is impossible because $C_{i}$ is strictly decreasing in size and so must eventually be empty.
\end{proof}
\end{lemma}
As a corollary, we will use this result to obtain a bound for $h(A)$ when $\min A$ is sufficiently large. Unfortunately, the exponent $1/(k-1)$ in Lemma \ref{large_min_split} means that this bound grows very fast, but we are still able to compute $h(3)$ and $h(4)$ exactly, and then we will bound $h(k)$ for larger $k$ by a different argument. The bounds we get for $h(5)$ will be much weaker than what should be true, but computing the exact value of $h(5)$ would seem to require an improvement in Lemma \ref{large_min_split}. Unfortunately, Lemma \ref{large_min_split} is sharp in general, as seen by considering some large $x$, primes $x+o(x)\ge p_{1}>p_{2}>\cdots>p_{k}\ge x$ and $a_{i}=\prod_{j\ne i}p_{j}$. In this case, $a_{i}/(a_{i},a_{j})=p_{j}\le x+o(x)\le a_{1}^{1/(k-1)}(1+o(1))$ for all $i\ne j$. It is, however, possible to improve the lemma for \emph{almost all} $A$, but we will not go into this as it is not necessary for the proof of Theorem \ref{pilz_large_n}.
\begin{cor}\label{large_min_bound}
Suppose that $k\ge 2$, $\gcd(A)=1$ and $\min A\ge 2$. Then there exist $1\le i\le k-1$ and $B,C$ disjoint such that, $|B|=i$, $|C|=k-i$, $A=B\cup C$ and
\[h(A)\ge \min_{1\le i\le k-1} \left(h(B)+h(C)-\frac{2i(k-i)}{(\min A)^{1/(k-1)}}\right).\]
\begin{proof}
This follows immediately from Lemma \ref{large_min_split} and Lemma \ref{partition_bound}.
\end{proof}
\end{cor}

The following lemma combines the results of the section so far into a form that can be easily implemented as an algorithm for computing $h(k)$. For a given $k$ and $H$, the lemma narrows down the possible exceptions to the inequality $h(A)\ge H$ to a finite list that can be checked.
\begin{lemma}\label{hk_algorithm}
Let $k\ge 2$. For $1\le i\le k/2$ suppose that $h^{\ast}(i)\in \bR^{+}$ is such that for all $1\le i\le k-1$ and $B,C\subseteq \bN$ with $|B|=i$, $|C|=k-i$, and $\gcd(B)=\gcd(C)=1$, we have  $h(B)+h(C)\ge h^{\ast}(i)$ unless $(B,C)$ belong to some exceptional set of pairs which we call $\mathcal{B}_{i}$.

Then for any $A\subseteq \bN$ with $|A|=k$, and any
\[1\le H<\min_{1\le i\le k-1} (h(i)+h(k-i)),\]
if $h(A)<H$ then $A$ must be of one of the following forms.
\begin{itemize}
    \item There are some $1\le i\le k/2$, $(B,C)\in \mathcal{B}_{i}$ and $x,y\in \bN$ coprime such that $A=(x\cdot B)\cup (y\cdot C)$ and
    \begin{equation}\label{hk_algorithm_first_case}x\le \frac{2 i(k-i)}{h(i)+h(k-i)-H}\max C,\quad y\le \frac{2 i(k-i)}{h(i)+h(k-i)-H}\max B.\end{equation}
    \item Writing $A=\{a_{1},\dots,a_{k}\}$ with $a_{1}<\cdots<a_{k}$, we have
    \begin{equation}\label{hk_algorithm_min_bound}a_{1}\le \max_{1\le i\le k-1}\left(\frac{2i(k-i)}{h^{\ast}(i)-H}\right)^{k-1}\end{equation}
    and for all $1\le i\le k-1$,
    \begin{equation}\label{hk_algorithm_ratio_bound}\frac{a_{i+1}}{a_{i}}<3\left(1-\frac{H-1}{h^{\ast}(i)-1}\right)^{-1}.\end{equation}
    Also, every $a\in A$ must be
    \[\max_{1\le i\le k-1}\left(\frac{2i(k-i)}{h^{\ast}(i)-H}\right)\]
    smooth.
\end{itemize}
\begin{proof}
Suppose that $1\le i\le k/2$ and $A=(x\cdot B)\cup (y\cdot C)$ for some $(B,C)\in \mathcal{B}_{i}$ and $x,y\in \bN$. Then $\gcd(A)=\gcd(B)=\gcd(C)=1$, so $(x,y)=1$. If $b\in B$ and $c\in C$ then $(bx,cy)\mid bc$, so
\[\frac{(bx,cy)}{\max\{bx,cy\}}\le \min\left\{\frac{c}{x},\frac{b}{y}\right\}.\]
Then by Lemma \ref{partition_bound},
\[H>h(A)\ge h(B)+h(C)-2 i(k-i)\min\left\{\frac{\max C}{x},\frac{\max B}{y}\right\},\]
which rearranges to \eqref{hk_algorithm_first_case}, so we are in the first case.

We may now suppose that whenever $A=(x\cdot B)\cup (y\cdot C)$ with $|B|=i$, $|C|=k-i$ and $\gcd(B)=\gcd(C)=1$, we must have $(B,C)\notin \mathcal{B}_{i}$, so $h(B)+h(C)\ge h^{\ast}(i)$. For $k/2<i\le k-1$, define $h^{\ast}(i)=h^{\ast}(k-i)$. Then by Lemma \ref{ratio_split}, it follows that for all $1\le i\le k-1$,
\[H>h(A)\ge 1+(h^{\ast}(i)-1)\left(1-3\frac{a_{i}}{a_{i+1}}\right),\]
and by Corollary \ref{large_min_bound}, it follows that for some $1\le i\le k-1$,
\[H>h(A)\ge h^{\ast}(i)-2\frac{i(k-i)}{a_{1}^{1/(k-1)}}.\]
Rearranging these, we get \eqref{hk_algorithm_ratio_bound} and \eqref{hk_algorithm_min_bound} respectively. Smoothness follows from Corollary \ref{large_prime_split} and so we must be in the second case.
\end{proof}
\end{lemma}

Lemma \ref{hk_algorithm} is only helpful in practice when $k\le 4$, but in these cases we can compute $h(k)$ exactly.
\begin{lemma}\label{h4_size}
We have $h(1)=h(2)=1$, $h(3)=4/3$ and $h(4)=5/3$.
\begin{proof}
Clearly $h(\{a\})=1$ for any $a\in \bN$. Also if $a<b$ and $(a,b)=1$ then $h(\{a,b\})=2-2/b$, so $h(2)=1$ with the minimum occurring only at $A=\{a,2a\}$, $a\in \bN$.

For $k=3$ we use Lemma \ref{hk_algorithm} with $\mathcal{B}_{1}=\{(\{1\},\{a,b\}):1\le a<b\le 5, \:(a,b)=1\}$. We may take $h^{\ast}(1)=1+2-2/5=13/5$. We also take $H=11/6$. In the first case of the lemma, $B=\{1\}$, $C=\{a,b\}$ with $b\le 5$, $(a,b)=1$ and $x\le 12\max C$, $y\le 12$. After discarding the cases where $x\cdot B$ and $y\cdot C$ are not disjoint, there are $12781$ cases to check. In the second case, there are $3421$ cases to check. This can easily be done with a computer, and we find that there are exactly $34$ cases where $h(A)<11/6$, all of which have $\max A\le 11$. For each of the $34$ exceptions, $h(A)\ge 4/3$, with equality when $A=\{1,2,3\}$, so $h(3)=4/3$.

For $k=4$, we use Lemma \ref{hk_algorithm} again. Let $\mathcal{B}_{1}$ consist of the pairs $(\{1\},C)$ with $C$ being one of the $34$ exceptional sets with $|C|=3$ and $h(C)<11/6$ and let $\mathcal{B}_{2}$ consist of the pairs $(\{b_{1},b_{2}\},\{c_{1},c_{2}\})$ with $1\le b_{1}<b_{2}$, $1\le c_{1}<c_{2}$, $(b_{1},b_{2})=(c_{1},c_{2})=1$ and either $b_{2}=2$, $c_{2}\le 19$ or $b_{2}=3$, $c_{2}\le 4$. With this choice of $\mathcal{B}_{1},\mathcal{B}_{2}$, we may take $h^{\ast}(1)=1+11/6=17/6$ and $h^{\ast}(2)=29/10$. Then we let $H=5/3$ and find that there are $1166230$ cases to check of the form $(x\cdot B)\cup (y\cdot C)$ and $772553$ cases of the second type.
\end{proof}
\end{lemma}

For larger $k$ we will use a different method. As in the proof of Theorem \ref{kn/logk_bound}, we will combine Corollary \ref{large_prime_split} with a lower bound that holds for smooth $A$. This lower bound will follow from a more general bound, which we state below. Given a subset $\mathcal{B}\subseteq \mathcal{P}$ of primes, we define
\[[n]_{\mathcal{B}}=\{m\in [n]:\forall p\in \mathcal{P},\: p\mid m\implies p\in \mathcal{B}\}.\]
Given $A\subseteq \bN$, we also let $\mathcal{B}_{A}$ be the set of all primes which divide $a$ for some $a\in A$.

\begin{lemma}\label{few_prime_factor_bound_improved}
Let $m\ge 1$ and $A\subseteq \bN$ be finite and let $\mathcal{B}$ be a finite set of primes such that $\mathcal{B}_{A}\subseteq \mathcal{B}$. Then
\begin{equation}\label{prime_prod_lower_bound_improved} h(A)\ge \sum_{r=1}^{\infty}\frac{|A\ast [r]_{\mathcal{B}}|}{r(r+1)}\prod_{p\in \mathcal{B}}\left(1-\frac{1}{p}\right).\end{equation}
\begin{proof}
Fix $m\ge 1$. For $r,n\in \bN$, let $S_{r}(n)$ be the set of all $x\in \bZ$ with $n/(r+1)<x\le n/r$ and $p\nmid x$ for all $p\in \mathcal{B}$ and let
\[S(n)=\bigcup_{r=1}^{m} S_{r}(n).\]
Clearly for any fixed $n$, these sets $S_{r}(n)$ for $r\in \{1,2,\dots, m\}$ are pairwise disjoint. Because $\mathcal{B}_{A}\subseteq \mathcal{B}$, we also have that for any $1\le i\le k$ and $x\in S(n)$, $(x,a_{i})=1$.

For each $1\le r\le m$ and $x\in S_{r}(n)$, let $T(x)=x\cdot [r]_{\mathcal{B}}$. We claim that for any $x\in S(n)$, $A\ast T(x)\subseteq A\ast [n]$ and that for any two $x,y\in S(n)$ with $x\ne y$, $A*T(x)$ and $A*T(y)$ are disjoint. It will then follow that
\begin{equation}\label{partition_AstarTx} |A\ast [n]|\ge \sum_{x\in S(n)} |A\ast T(x)|.
\end{equation}
To prove the claim, suppose that $x\in S_{r}(n)$, $u\in [r]_{\mathcal{B}}$, and $y\in [n]$ are such that $a_{i}ux=a_{j}y$ for some $a_{i},a_{j}\in A$. Then, since $(x,a_{j})=1$, we must have $x\mid y$, and so there is some $v\in \bN$ such that $y=vx$. Since $y\in [n]$ and $x>n/(r+1)$, we must have $v<r+1$ and so $v\in [r]$. From the fact that $v\mid a_{i}u$, we may deduce that every prime divisor of $v$ is in $\mathcal{B}$ and so $v\in [r]_{\mathcal{B}}$. This tells us that if $z\in A\ast T(x)$ then any representation of $z$ as $a_{i}y$, with $y\in [n]$ must satisfy $y\in T(x)$. The claim then follows, and hence we have \eqref{partition_AstarTx}.

Now, clearly $|A\ast T(x)|=|A\ast [r]_{\mathcal{B}}|$ for any $x\in S_{r}(n)$, and it is also not hard to show, via the inclusion-exclusion principle, that
\[\lim_{n\rightarrow \infty}\frac{S_{r}(n)}{n}=\left(\frac{1}{r}-\frac{1}{r+1}\right)\prod_{p\in \mathcal{B}}\left(1-\frac{1}{p}\right).\]
It then follows that
\[h(A)\ge \sum_{r=1}^{m}\frac{|A\ast [r]_{\mathcal{B}}|}{r(r+1)}\prod_{p\in \mathcal{B}}\left(1-\frac{1}{p}\right).\]
As $m\in \bN$ was arbitrary, \eqref{prime_prod_lower_bound_improved} follows.
\end{proof}
\end{lemma}
As a special case, with $\mathcal{B}=\mathcal{B}_{A}$ and discarding all terms with $r\ge2$, we have
\begin{equation}\label{prime_prod_lower_bound} h(A)\ge \frac{k}{2}\prod_{p\mid a_{1}\cdots a_{k}}\left(1-\frac{1}{p}\right).\end{equation}
This can be strengthened to the following.

\begin{cor}\label{prime_prod_with_g}
Suppose that $\gcd(A)=1$ and let $q$ be the largest prime divisor of $a_{1}\cdots a_{k}$ and $q^{+}$ the smallest prime such that $q^{+}>q$. Then
\[h(A)\ge \left(\frac{k}{2}+2\sum_{r=2}^{q-1}\frac{g(r)}{r(r+1)}+\sum_{r=q}^{q^{+}-1}\frac{g(r)}{r(r+1)}\right)\prod_{p\le q}\left(1-\frac{1}{p}\right).\]
\begin{proof}
We apply Lemma \ref{few_prime_factor_bound_improved} with $\mathcal{B}$ the set of all primes which are at most $q$, and note that in this case, $[r]_{\mathcal{B}}=[r]$ for all $r\le q^{+}-1$. It suffices to prove that:
\begin{enumerate}
\item $|A\ast [1]|=|A|=k$,
\item $|A\ast [r]|\ge 2g(r)$ for $2\le r<q$ and
\item $|A\ast [r]|\ge g(r)$ for $q\le r<q^{+}$.
\end{enumerate}
These are all immediate except for (2). Suppose that $2\le r<q$. By Lemma \ref{rough_decomp}, we can decompose $A$ as a disjoint union
\[A=\bigcup_{i=1}^{m} b_{i}\cdot C_{i},\]
where $b_{1},\dots, b_{m}$ are distinct and $r$-rough, and each $C_{i}$ is a nonempty set of $r$-smooth numbers. Lemma \ref{rough_decomp} also tells us that
\[|A\ast [r]|\ge mg(r).\]
We have $q\mid a_{i}$ for some $i$, and $q>r$, so we must have $q\mid b_{j}$ for some $j$. The fact that $\gcd(A)=1$ means that we must also have some $b_{\ell}$ such that $q\nmid b_{\ell}$ and so $m\ge 2$. The result follows.
\end{proof}
\end{cor}

\begin{defn}
For $p$ a prime and $n\in \bN$, we let $\alpha_{p}(n)$ be the largest $\alpha\in \bZ$ such that $p^{\alpha}\le n$. Then we define $g^{\dag}(n)$ by
\[g^{\dag}(n)=\begin{cases}n&\text{if $1\le n\le 8$}\\ \sum_{p\le n}g^{\dag}\left(\lfloor n/p^{\alpha_{p}(n)}\rfloor\right) &\text{if $n\ge 9$}\end{cases}.\]
\end{defn}

\begin{lemma}\label{g_gdag_bound}
For all $n\in \bN$, $g(n)\ge g^{\dag}(n)$.
\begin{proof}
For $n\le 8$, this is precisely Theorem \ref{g_le8} and for $n\ge 9$, it follows immediately from \cite[Proposition 1]{ps2011}.
\end{proof}
\end{lemma}

\begin{lemma}\label{hk_algorithm2}
Let $p_{i}$ be the $i$th prime and for $k\in \bN$, let $h^{\dag}(k)$ be such that $h(k)\ge h^{\dag}(k)$. For $k,i\in \bN$, let
\[a_{k}(i)=\min_{1\le j\le k-1} \left(h^{\dag}(j)+h^{\dag}(k-j)-\frac{2 j(k-j)}{p_{i+1}}\right),\]
and
\[b_{k}(i)=\min_{1\le j\le i}\left(\frac{k}{2}+2\sum_{r=2}^{p_{j}-1}\frac{g^{\dag}(r)}{r(r+1)}+\sum_{r=p_{j}}^{p_{j+1}-1}\frac{g^{\dag}(r)}{r(r+1)}\right)\prod_{m\le j}\left(1-\frac{1}{p_{m}}\right).\]
Then for any $k,i\in \bN$,
\[h(k)\ge \min\{a_{k}(i),b_{k}(i)\}.\]
\begin{proof}
Let $A\subset \bN$ with $|A|=k$ and suppose without loss of generality that $\gcd(A)=1$. Let $p_{j}$ be the largest prime that divides some $a\in A$. If $j>i$ then $h(A)\ge a_{k}(i)$ by Corollary \ref{large_prime_split}. If $j\le i$ then it follows from Corollary \ref{prime_prod_with_g} and Lemma \ref{g_gdag_bound} that $h(A)\ge b_{k}(i)$.
\end{proof}
\end{lemma}
In the above lemma, the sequence $a_{k}(i)$ is increasing in $i$, and $b_{k}(i)$ is decreasing in $i$, so it is easy to compute the $i$ such that $\min\{a_{k}(i),b_{k}(i)\}$ is maximal. Using the results of Lemma \ref{h4_size} we, may thus iteratively compute a sequence $h^{\dag}(k)$ such that $h(k)\ge h^{\dag}(k)$ for all $k$. We tabulate some values of $h^{\dag}(k)$ below, truncated (i.e. rounded down) to $4$ decimal places, along with the value of $i$ that is used in the application of Lemma \ref{hk_algorithm2}.
\begin{table}[ht]
\caption{Lower bounds on $h(k)$}
\centering
\begin{tabular}{ |c|c|c|c|c|c|c| }
 \cline{1-3} \cline{5-7}
 $k$ & $h^{\dag}(k)$ & $i$ & \hspace{2cm} & $k$ & $h^{\dag}(k)$ & $i$\\ \cline{1-3} \cline{5-7}
5 & 1.1455\dots & 3 & \hspace{2cm} & 19 & 1.6386\dots & 41\\
6 & 1.2304\dots & 4 & \hspace{2cm} & 20 & 1.6674\dots & 45\\
7 & 1.2060\dots & 7 & \hspace{2cm} & 21 & 1.7133\dots & 47\\
8 & 1.2767\dots & 8 & \hspace{2cm} & 22 & 1.7366\dots & 52\\
9 & 1.3133\dots & 10 & \hspace{2cm} & 23 & 1.7745\dots & 55\\
10 & 1.3194\dots & 14 & \hspace{2cm} & 24 & 1.7997\dots & 60\\
11 & 1.3650\dots & 16 & \hspace{2cm} & 25 & 1.8333\dots & 64\\
12 & 1.3926\dots & 19 & \hspace{2cm} & 26 & 1.8725\dots & 67\\
13 & 1.4268\dots & 22 & \hspace{2cm} & 27 & 1.9008\dots & 72\\
14 & 1.4606\dots & 25 & \hspace{2cm} & 28 & 1.9329\dots & 76\\
15 & 1.4916\dots & 28 & \hspace{2cm} & 29 & 1.9626\dots & 81\\
16 & 1.5454\dots & 30 & \hspace{2cm} & 30 & 1.9913\dots & 86\\
17 & 1.5799\dots & 33 & \hspace{2cm} & 31 & 2.0232\dots & 90\\
18 & 1.6146\dots & 36 & \hspace{2cm} & 32 & 2.0566\dots & 94\\
 \cline{1-3} \cline{5-7}
\end{tabular}
\label{hk_table}
\end{table}
This gives the following lemma.
\begin{lemma}\label{hk_bound_med_k}
For $5\le k\le 32$, we have
\[h(k)\ge h^{\dag}(k),\]
where $h^{\dag}(k)$ is as defined in the above table.
\end{lemma}
Note that all of the values in the above table are at least $8/7=1.1428\dots$, as required for Theorem \ref{hA_uniform_lower_bound}.

\begin{lemma}\label{constant_bound_large_k}
For all $k\ge 33$, we have $h(k)\ge 2$.
\begin{proof}
We proceed by induction on $k$. Suppose that $k\ge 33$ and $h(i)\ge 2$ for all $33\le i\le k-1$. If $k\le 53$ then by Lemma \ref{hk_algorithm2} and a computer calculation, we can check that $h(k)\ge 2$ as required. Suppose then that $k\ge 54$ and suppose for a contradiction that $A\subseteq \bN$ such that $|A|=k$ and $h(A)<2$.

From \cite[Theorem 7]{rs1962}, if $x\ge 285$, then
\begin{equation}\label{RS_thm7}\prod_{p\le x}\left(1-\frac{1}{p}\right)\ge \frac{e^{-\gamma}}{\log x}\left(1-\frac{1}{2(\log x)^{2}}\right),\end{equation}
where $\gamma$ is the Euler-Mascheroni constant. 

Let $q$ be the largest prime which divides some $a\in A$ and $q_{1}:=\max\{q,285\}$. Then
\[e^{-\gamma}\left(1-\frac{1}{2(\log q_{1})^{2}}\right)>\frac{21}{38},\]
so
\[\prod_{p\mid a_{1}\cdots a_{k}} \left(1-\frac{1}{p}\right) \ge \prod_{p\le q_{1}}\left(1-\frac{1}{p}\right) \ge \frac{21}{38} \cdot \frac{1}{\log q_{1}}.\]
Using Lemmas \ref{hk_bound_med_k}, \ref{h4_size} and the induction hypothesis, we can check that $h(i)+h(k-i)\ge 3$ for all $1\le i\le k-1$. Therefore, by Corollary \ref{large_prime_split}, we must have
\[q\le 2\left\lfloor\frac{k^{2}}{4}\right\rfloor\le \frac{k^{2}}{2},\]
so
\[\log q_{1}\le 2\log k-\log 2.\]
Where we have used the fact that $k\ge 54$ and so $k^{2}/2>285$.

Now, by equation \eqref{prime_prod_lower_bound}, we have
\[h(A)> \frac{21}{76}\cdot \frac{k}{2\log k-\log 2}.\]
The function $x/(2\log x-\log 2)$ is increasing for all $x>e\sqrt{2}$, so it suffices to note that
\[\frac{21}{76}\cdot \frac{54}{2\log 54 -\log 2}>2.\]
\end{proof}
\end{lemma}

Theorem \ref{hA_uniform_lower_bound} now follows by combining Lemmas \ref{h4_size}, \ref{hk_bound_med_k} and \ref{constant_bound_large_k}.

\section{Verifying the conjecture for large $n$}\label{large_n_section}
The goal of this section is to prove Theorem \ref{pilz_large_n}. The main idea of the proof is the same slicing argument used in Section \ref{small_n_section}. However, in those proofs, we used some ad hoc arguments which do not obviously generalise to arbitrary $n$. In the general case we replace these ad hoc arguments with lower bounds for sets of the form $(A_{1}\ast ([n]\setminus T))\Delta\cdots \Delta (A_{r}\ast ([n]\setminus T)^{\ast r})$ where the $A_{i}$ are relatively small sets and $T\subset (n/2,n]$ is a set of primes. This will be the focus of Lemmas \ref{A_star_n_minus_T_bound} to \ref{star_power_2_intersection}. We begin with the simplest case: $r=1$.

Recall the notation $A_{p}^{(i)}=\{a\in A: \nu_{p}(a)=i\}$ from section \ref{small_n_section}, which we will use throughout this section.
\begin{lemma}\label{A_star_n_minus_T_bound}
Let $n\in \bN$, $0<\eps<1$ and $T\subseteq (n/2,n]$ a set of primes with $|T|\le n^{1-\eps}$. Also let $A\subseteq \bN$ such that $2\le k=|A|$ and $A\ne \{a,2a\}$ for all $a\in \bN$. Then whenever
\[n\ge \max\left\{k^{2}, 2\cdot \left(3^{80}-321\right), (40020\log k-20010\log 2)^{1/\eps},4480^{1/\eps}\right\},\]
we have
\[\left|A\ast \left([n]\setminus T\right)\right|\ge n.\]
\begin{proof}
Suppose first that $k\le 80$ and let $A\subset \bN$ with $|A|=k$, where without loss of generality, $\gcd(A)=1$. Given some $p\in T$, we may show that
\[A_{p}^{(0)}\ast ([n]\setminus T)\subseteq A\ast ([n]\setminus T).\]
Hence we may suppose in all that follows that no $a\in A$ is divisible by any $p\in T$.

Given $1\le i_{1}<\cdots <i_{t}\le k$, we claim that when $t\ge 2$,
\[(a_{i_{1}}\cdot ([n]\setminus T))\cap \cdots \cap (a_{i_{t}}\cdot ([n]\setminus T))=(a_{i_{1}}\cdot [n])\cap \cdots \cap (a_{i_{t}}\cdot [n]).\]
Indeed, if $a_{i}x=a_{j}y$ for some $i\ne j$, $x\in T$ and $y\in [n]$, then since $x\in T$ is prime, either $(x,y)=1$ or $x=y$. The second case is impossible because then $a_{i}=a_{j}$, so $(x,y)=1$ and then $x\mid a_{j}$. But we supposed that no $a\in A$ is divisible by any $p\in T$, so this is a contradiction.

We also have, for all $1\le i_{1}<\cdots <i_{t}\le k$,
\[\left|(a_{i_{1}}\cdot [n])\cap \cdots \cap (a_{i_{t}}\cdot [n])\right|=\left\lfloor \frac{a_{i_{1}}n}{[a_{i_{1}},\dots,a_{i_{t}}]}\right\rfloor,\]
and for all $1\le i\le k$,
\[|a_{i}\cdot ([n]\setminus T)|=n-|T|.\]
Then, by the inclusion-exclusion principle and \eqref{exact_formula_eq},
\[\left|A\ast \left([n]\setminus T\right)\right|>h(A)n-k|T|-\sum_{t=2}^{\lceil k/2\rceil}\binom{k}{2t-1}2^{2t-2}.\]
Now,
\begin{align*}\sum_{t=2}^{\lceil k/2\rceil}\binom{k}{2t-1}2^{2t-2}&=\frac{1}{4}\left(\sum_{t=0}^{k}\binom{k}{t}2^{t}-\sum_{t=0}^{k}\binom{k}{t}(-2)^{t}\right)-k\\&=\frac{3^{k}-(-1)^{k}}{4}-k.\end{align*}
Note also that $k|T|\le k n^{1-\eps}$, so we have
\[\left|A\ast \left([n]\setminus T\right)\right|> h(A)n-kn^{1-\eps}-\frac{3^{k}-(-1)^{k}-4k}{4}.\]
By Theorem \ref{hA_uniform_lower_bound}, $h(A)\ge 8/7$, so with the assumptions $k\le 80$, $n\ge 4480^{1/\eps}$ and $n\ge 2\cdot \left(3^{80}-321\right)$, we have the desired bound.

Now suppose that $k\ge 81$ and that the result holds for all smaller values of $k$. Let $A=\{a_{1},\dots,a_{k}\}$ for some $a_{1}<\cdots<a_{k}$ with $\gcd(A)=1$ and let $p$ be the largest prime divisor of $a_{1}\cdots a_{k}$. By Corollary \ref{large_prime_split}, there are $1\le m\le k-1$ and $B$ and $C$ such that $A=B\cup C$, $|B|=m$, $|C|=k-m$, and
\[\left| A\ast \left([n]\setminus T\right)\right|\ge \left| B\ast \left([n]\setminus T\right)\right|+\left| C\ast \left([n]\setminus T\right)\right|-2n\frac{m(k-m)}{p}.\]
If neither of $B$ or $C$ is of the form $\{a\}$ or $\{a,2a\}$ for some $a\in \bN$, then from the induction hypothesis, the desired bound follows provided $p\ge k^{2}/2$. If $B$ is of the form $\{a\}$ or $\{a,2a\}$, then $| B\ast ([n]\setminus T)|\ge n-2n^{1-\eps}$, so we have
\[\left| A\ast \left([n]\setminus T\right)\right|\ge 2n-2n^{1-\eps}-\frac{4n(k-2)}{p},\]
and from the assumption that $n\ge 4480^{1/\eps}$ we again get the desired bound when $p\ge k^{2}/2>8(k-2)$.

We may therefore suppose that $p<k^{2}/2\le n/2$. Then by \eqref{A_star_n_smooth_bound_explicit}, we have
\[\left| A\ast \left([n]\setminus T\right)\right|\ge \frac{kn}{20\log k-10\log 2}-kn^{1-\eps}.\]
Now, since we assumed that $n\ge (40020\log k-20010\log 2)^{1/\eps}$, we find that
\[\left| A\ast \left([n]\setminus T\right)\right|\ge \frac{200}{2001}\cdot \frac{kn}{2\log k-\log 2}.\]
The function $x/(2\log x-\log 2)$ is increasing for $x\ge \sqrt{2}e$, so it suffices to verify that
\[\frac{200}{2001}\cdot \frac{81}{2\log 81-\log 2}>1.\]
\end{proof}
\end{lemma}
We now move on to understanding sets of the form $A\ast S^{\ast u}$ for $u>1$. When $u$ is a power of $2$, this takes a relatively simple form.
\begin{lemma}\label{power_2_power}
Let $v\in \bZ_{\ge 0}$ and $u=2^{v}$ and suppose that $A,S\subseteq \bN$ are finite sets. Then there exists a unique $\mathcal{A}\subseteq \bN$ of $u$-free integers (that is, integers which are not divisible by $p^{u}$ for any prime $p$) and $B_{a}\subseteq \bN$ for each $a\in \mathcal{A}$ such that
\[A=\bigcup_{a\in \mathcal{A}}a\cdot \left\{b^{u}:b\in B_{a}\right\}.\]
Furthermore, $A\ast S^{\ast u}$ is given by the disjoint union
\begin{equation}\label{ast_even_power_formula}A\ast S^{\ast u}=\bigcup_{a\in \mathcal{A}} a\cdot \left\{b^{u}:b\in B_{a}\ast S\right\},\end{equation}
so that in particular,
\[\left|A\ast S^{\ast u}\right|=\sum_{a\in \mathcal{A}}\left|B_{a}\ast S\right|.\]
\begin{proof}
Every integer has a unique decomposition as a product of a $u$-free part and a $u$th power. Letting $\mathcal{A}$ be the set of all $u$-free parts of the elements of $A$ gives us the desired decomposition. Before we prove \eqref{ast_even_power_formula}, we prove the special case
\[S^{\ast u}=\{x^{u}:x\in S\}\]
for any finite $S\subseteq \bN$. This will follow by induction on $v$ from the identity
\[S\ast S=\{x^{2}: x\in S\}.\]
The proof of this identity uses the same idea as the proof of \cite[Theorem 5.1]{km2025}. For each $z\in \bN$, consider the representations $z=xy$ with $x,y\in S$. Those representations with $x\ne y$ form pairs because $xy=yx$, so these have no effect on the parity of $r_{S,S}(z)$. Therefore, $r_{S,S}(z)$ is odd if and only if $z=x^{2}$ for some $x\in S$.

It follows that
\[A=\bigcup_{a\in \mathcal{A}} a\cdot B_{a}^{\ast u}\]
and
\[A\ast S^{\ast u}=\bigdelta_{a\in \mathcal{A}} a\cdot \left(B_{a}^{\ast u}\ast S^{\ast u}\right)=\bigdelta_{a\in \mathcal{A}} a\cdot\left\{b^{u}:b\in B_{a}\ast S\right\}.\]
The sets $a\cdot \{x^{u}:x\in B_{a}\ast S\}$ are disjoint, so this is in fact a disjoint union, which completes the proof.
\end{proof}
\end{lemma}

When $u$ is not a power of $2$, things are more difficult.
\begin{defn}For a natural number $u$, we let $\beta(u)$ denote the number of $1$s in the binary digit expansion of $u$.
\end{defn}
For $u\in \bN$, it can be shown that there is some $c_{u}>0$ such that $|A\ast ([n]\setminus T)^{\ast u}|\ge c_{u}n^{\beta(u)}$ for all $A\ne \emptyset$. When we take the symmetric difference with some other $A'\ast ([n]\setminus T)^{\ast v}$, say, we need to bound the size of the intersection. This can be done without too much difficulty when $n$ is large in terms of $\beta(u)$ and $\beta(v)$, but this is not quite good enough, because $\beta(u)$ and $\beta(v)$ may be arbitrarily large. We instead restrict to a special subset of $[n]$ which will allow us to more easily control the intersection with other sets.
\begin{lemma}\label{mult_diophantine_prime_solutions}
Suppose $n\in \bN$ and $\mathcal{Q}\subseteq (n/2,n]$ is a nonempty set of primes. Let $a,b,r,s\in \bN$ and $u_{1},\dots,u_{r},v_{1},\dots,v_{s}\in \bZ_{\ge 0}$ be fixed and suppose that $a\ne b$, $r\ge s$, $0\le u_{1}<\cdots <u_{r}$ and $0\le v_{1}<\cdots <v_{s}$. Then the number of solutions $x_{1},\dots,x_{r}\in \mathcal{Q}$, $y_{1},\dots,y_{s}\in [n]$ to the equation
\begin{equation}\label{mult_diophantine}a x_{1}^{2^{u_{1}}} \cdots x_{r}^{2^{u_{r}}}=b y_{1}^{2^{v_{1}}}\cdots y_{s}^{2^{v_{s}}}\end{equation}
is at most
\[2|\mathcal{Q}|^{r-1}.\]
Furthermore, if we instead assume that $a=b$, then the only solutions to \eqref{mult_diophantine} are when $r=s$, $u_{i}=v_{i}$ and $x_{i}=y_{i}$ for all $1\le i\le r$.
\begin{proof}
We may suppose without loss of generality that $(a,b)=1$. We first prove the second part of the lemma, where $a=b=1$. Then \eqref{mult_diophantine} becomes
\[x_{1}^{2^{u_{1}}} \cdots x_{r}^{2^{u_{r}}}=y_{1}^{2^{v_{1}}}\cdots y_{s}^{2^{v_{s}}}.\]
Consider $p\in \mathcal{Q}$. Recalling that $\nu_{p}$ is the $p$-adic valuation of a number, we have that
\[\nu_{p}\left(y_{1}^{2^{v_{1}}}\cdots y_{s}^{2^{v_{s}}}\right)=\sum_{j=1}^{s}2^{v_{j}}\nu_{p}\left(y_{j}\right)=\sum_{i=1}^{r}2^{u_{i}}\nu_{p}(x_{i}).\]
Since $p>n/2$, we must have $\nu_{p}(x_{i}),\nu_{p}(y_{j})\in \{0,1\}$ for all $1\le i\le r$, $1\le j\le s$. Therefore, $\sum\limits_{i:\ x_i=p} 2^{u_i}=\sum\limits_{j:\ y_j=p} 2^{v_j}$, where the two sides are the binary representations of the same number. So $p$ appears in $\{x_1,\dots,x_r\}$ and $\{y_1,\dots,y_s\}$ with the same multiplicity, and the sets of the corresponding exponents also must coincide. Using the fact that $r\geq s$, we get that $r=s$, $u_i=v_i$, $x_i=y_i$ for all $i$, as required.

Now suppose that $a\ne b$. Observe that we may assume that each prime factor of $ab$ is from $\mathcal Q$. Indeed, let the largest divisor of the numbers $a,b,y_1,\dots,y_s$ containing prime factors only from $\mathcal Q$ be denoted by $\tilde{a},\tilde{b},\tilde{y_1},\dots,\tilde{y_s}$, respectively. Then  
$$\tilde{a} x_{1}^{2^{u_{1}}} \cdots x_{r}^{2^{u_{r}}}=\tilde{b} \tilde{y}_{1}^{2^{v_{1}}}\cdots \tilde{y}_{s}^{2^{v_{s}}}$$
must hold. This is an equation of the same type satisfying the extra constraint, except the case when $\tilde{a}=\tilde{b}=1$. However, $\tilde{a}=\tilde{b}=1$ would imply that $u_i=v_i,x_i=\tilde{y}_i=y_i$ for every $i$, contradicting the assumption that $a\ne b$. Also, a solution $x_1,\dots,x_r,\tilde{y}_1,\dots,\tilde{y}_s$ determines $y_1,\dots,y_s$, as well. Therefore, from now on, we may further assume that each prime divisor of $ab$ is from $\mathcal Q$.

For the sake of contradiction assume that the statement is false. Let us take a counterexample where $r+s$ is minimal, and among these $a+b+\sum\limits_{i=1}^r u_i+\sum\limits_{j=1}^s v_j$ is minimal. Note that $x_1,\dots,x_r$ determine $y_1,\dots,y_s$ according to the already applied argument using the base-2 representations of the $p$-adic valuations.

If $0<\min(u_1,v_1)=t$, then $a$ and $b$ must be perfect $2^t$-th powers, so we may replace them by their $2^t$-th roots and decrease every $u_i,v_j$ by $t$. Hence, $\min(u_1,v_1)=0$ can be assumed.

If $u_1=0<v_1$, then $abx_1$ must be a perfect square. This can happen only if the square-free part of $ab$ is a prime $p\in \mathcal Q$ and $x_1=p$, so $x_1$ is uniquely determined, giving at most $|\mathcal Q|^{r-1}$ solutions. 

If $v_1=0<u_1$, then $aby_1$ must be a perfect square, so either the square-free part of $ab$ is a prime $p\in \mathcal Q$ and $y_1=p$, or $ab$ is a perfect square and $y_1=1$. Both cases contradict the minimality of the counterexample with only $r+(s-1)$ variables (replacing $b$ by $bp$ in the case when $y_1=p$).

Finally, let $u_1=v_1=0$, then $abx_1y_1$ must be a perfect square. The square-free part of $ab$ can be 1, $p$ or $pp'$ with primes $p,p'\in \mathcal Q$. If it is $p$ or $pp'$, then $x_1\in \{p,p'\}$, and we are done. If it is 1, then $x_1=y_1$. This again contradicts the minimality of the counterexample, completing the proof.
\end{proof}
\end{lemma}
\begin{cor}\label{power_n_intersection}
Let $a,b,u,v,n\in \bN$, $\mathcal{Q}\subseteq (n/2,n]$ a set of primes and $S\subseteq [n]$ such that $\mathcal{Q}\subseteq S$. If $\beta(u)\ge \beta(v)$ and $(a,u)\ne (b,v)$ then
\[\mathcal{Q}^{\ast u} \subseteq S^{\ast u},\]
\[\#\left(\mathcal{Q}^{\ast u}\right)=|\mathcal{Q}|^{\beta(u)}\]
and
\[\#\left(\left(a\cdot \mathcal{Q}^{\ast u}\right)\cap \left(b\cdot S^{\ast v}\right)\right)\le 2|\mathcal{Q}|^{\beta(u)-1}.\]
\begin{proof}
Let $r=\beta(u)$, $s=\beta(v)$. There are unique $u_{1}<\cdots<u_{r}$ and $v_{1}<\cdots<v_{s}$ such that
\[u=\sum_{i=1}^{r} 2^{u_{i}},\quad v=\sum_{i=1}^{s} 2^{v_{i}}.\]
Then by Lemma \ref{power_2_power} with $A=\{1\}$, for any finite set $T\subseteq \bN$,
\[T^{\ast u}=\left\{x^{2^{u_{1}}}: x\in T\right\}\ast \cdots \ast \left\{x^{2^{u_{r}}}:x\in T\right\}\]
and a similar equality holds for $T^{\ast v}$.

The second part of Lemma \ref{mult_diophantine_prime_solutions} tells us that each $x_{1}^{2^{u_{1}}}\cdots x_{r}^{2^{u_{r}}}$ with $x_{1},\ldots ,x_{r}\in \mathcal{Q}$ occurs exactly once as a product $y_{1}^{2^{u_{1}}}\cdots y_{r}^{2^{u_{r}}}$ with $y_{1},\dots,y_{r}\in S$, so
\[\mathcal{Q}^{\ast u} \subseteq S^{\ast u}.\]
The second equality then follows easily. The final inequality follows from the first part of Lemma \ref{mult_diophantine_prime_solutions}.
\end{proof}
\end{cor}

\begin{lemma}\label{A_chain_lower_bound}
Suppose that $A_{1},\dots,A_{r}\subseteq \bN$, $n\in \bN$ and $0<\eps\le 1/2$ with $\sum_{i=1}^{r}|A_{i}|\le n^{\eps}+1$. Let $s$ be the maximum of $\beta(i)$ among those $i$ such that $A_{i}\ne \emptyset$ and suppose that $n^{\eps}\ge 40\log n$. Then for all $T\subseteq [n]$ such that $|T|\le n^{1-\eps}$, we have
\begin{equation}\label{A_chain_lower_bound_eq}\#\left(\left(A_{1}\ast ([n]\setminus T)\right)\Delta \cdots \Delta \left(A_{r}\ast ([n]\setminus T)^{\ast r}\right)\right)\ge \left(\frac{n}{5\log n}\right)^{s}.\end{equation}
\begin{proof}
Let $j$ be such that $A_{j}\ne \emptyset$ and $\beta(j)=s$ and let $\mathcal{Q}$ be the set of all primes in $(n/2,n]\setminus T$. Fix some $a\in A_{j}$. Then by Corollary \ref{power_n_intersection},
\[\mathcal{Q}^{\ast j} \subseteq ([n]\setminus T)^{\ast j},\]
\[|\mathcal{Q}^{\ast j}|=|\mathcal{Q}|^{s}\]
and for all $1\le i\le r$ and $b\in A_{i}$ with $(a,j)\ne (b,i)$,
\[\#\left(\left(a\cdot \mathcal{Q}^{\ast j}\right)\cap \left(b\cdot ([n]\setminus T)^{\ast i}\right)\right)\le 2|\mathcal{Q}|^{s-1}.\]
Hence,
\[\#\left(\left(A_{1}\ast ([n]\setminus T)\right)\Delta \cdots \Delta \left(A_{r}\ast ([n]\setminus T)^{\ast r}\right)\right)\ge |\mathcal{Q}|^{s}-2|\mathcal{Q}|^{s-1}\sum_{i=1}^{r} |A_{i}|.\]
Now by \cite[Corollary 3]{rs1962}, for $n\ge 41$,
\[|\mathcal{Q}|\ge \frac{3n}{10\log n}-n^{1-\eps},\]
so recalling that $\sum_{i=1}^{r}|A_{i}|\le n^{\eps}+1$, we have \eqref{A_chain_lower_bound_eq} as long as
\[\frac{3n}{10\log n}-n^{1-\eps}-2n^{\eps}-2\ge \frac{n}{5\log n},\]
and we may check that this holds when $n^{\eps}\ge 40 \log n$.
\end{proof}
\end{lemma}
Lemma \ref{A_chain_lower_bound} handles all cases except when $A_{i}\ne \emptyset$ only for powers of two. The following lemma will allow us to handle this case as well.
\begin{lemma}\label{star_power_2_intersection}
Let $a,b,n\in \bN$, $S\subseteq [n]$ and $i,j\in \bZ$ with $0\le i<j$. Then
\[\left|\left(a\cdot S^{\ast 2^{i}}\right)\cap\left(b\cdot S^{\ast 2^{j}}\right)\right|\le n^{2^{i-j}}.\]
\begin{proof}
By Lemma \ref{power_2_power}, the quantity to be bounded is at most the number of solutions to the equation
\[ax^{2^{i}}=by^{2^{j}}\]
with $x,y\in [n]$. We may suppose without loss of generality that $(a,b)=1$. Then the equation implies that both $a$ and $b$ are perfect $2^i$-th powers: $a=\hat{a}^{2^i}$ and $b=\hat{b}^{2^i}$. The equation yields that $\hat{a}\cdot x=\hat{b}\cdot y^{2^{j-i}}$.
Write $x=x_{1}x_{2}^{2^{j-i}}$ with $x_{1}$ $2^{j-i}$-free.
Since $x_{1}$ is uniquely determined by $\hat{a}$ and $\hat{b}$, there are at most
\[\left\lfloor \left(\frac{n}{x_{1}}\right)^{2^{i-j}}\right\rfloor\leq n^{2^{i-j}}\]
choices of $x_{2}$. This completes the proof.
\end{proof}
\end{lemma}
\begin{proof}[Proof of Theorem \ref{pilz_large_n}]
Let $\mathcal{P}$ be the set of primes $p$ in the range $n/2<p\le n$, and for $\mathcal{Q}\subseteq \mathcal{P}$, we let $M(\mathcal{Q})$ be the largest number such that for every $A\subseteq \bN$ nonempty and finite, we have
\[\left|A\ast ([n]\setminus \mathcal{Q})\right|\ge M(\mathcal{Q}).\]
We claim that when $n$ is sufficiently large and $|\mathcal{Q}|\le   n^{4/5}$ then for any $q\in \mathcal{Q}$,
\begin{equation}\label{add_q_induction_step}M(\mathcal{Q}\setminus\{q\})\ge \min\left\{M(\mathcal{Q})+ n^{1/5}+1,n+2-2|\mathcal{Q}|\right\},\end{equation}
from which it follows by induction, starting with $|\mathcal{Q}|=\lfloor  n^{4/5}\rfloor$ as the base case, that
\[|A\ast [n]|\ge M(\emptyset)\ge n,\]
when $n$ is sufficiently large.

Let $q\in \mathcal{Q}$, $\mathcal{Q}'=\mathcal{Q}\setminus\{q\}$ and let $A\subseteq \bN$ be finite and nonempty.
Without loss of generality, $\gcd(A)=1$, so that $A_{q}^{(0)}\ne \emptyset$. Then by Lemma \ref{star_slice_formula},
\begin{align}\label{large_n_sliceformula}\nonumber \left|A\ast \left([n]\setminus \mathcal{Q}'\right)\right|&=\left|A_{q}^{(\nu_{q}(A))}\right|+\sum_{i=1}^{\nu_{q}(A)}\left|\left(A_{q}^{(i)}\ast \left([n]\setminus \mathcal{Q}\right)\right)\Delta \left(q\cdot A_{q}^{(i-1)}\right)\right|\\ &+\left|A_{q}^{(0)}\ast \left([n]\setminus \mathcal{Q}\right)\right|.\end{align}
Define $\mathcal{A}_{i}$ for $1\le i\le \nu_{q}(A)+1$, to be such that
\begin{equation}\label{calAdef}q^{i}\cdot \mathcal{A}_{i}=\left(A_{q}^{(i)}\ast \left([n]\setminus \mathcal{Q}\right)\right)\Delta \left(q\cdot A_{q}^{(i-1)}\right).\end{equation}
Note that we may have $\nu_{q}(A)=0$, in which case we simply have $\mathcal{A}_{1}=A$.

By the definition of $M$,
\[\left|A_{q}^{(0)}\ast \left([n]\setminus \mathcal{Q}\right)\right|\ge M(\mathcal{Q}),\]
so
\[\left|A\ast \left([n]\setminus \mathcal{Q}'\right)\right|\ge M(\mathcal{Q})+\sum_{i=1}^{\nu_{q}(A)+1} |\mathcal{A}_{i}|.\]
Therefore, if $\sum_{i=1}^{\nu_{q}(A)+1} |\mathcal{A}_{i}|\ge n^{1/5}+1$ then we are done, so in what follows, we suppose that
\[\sum_{i=1}^{\nu_{q}(A)+1} |\mathcal{A}_{i}|<n^{1/5}+1.\]
Rearranging \eqref{calAdef}, for $1\le i \le \nu_{q}(A)$, we have
\[q\cdot A_{q}^{(i-1)}=\left(q^{i}\cdot \mathcal{A}_{i}\right)\Delta \left(A_{q}^{(i)}\ast \left([n]\setminus \mathcal{Q}\right)\right).\]
Applying this inductively, we find that
\[A_{q}^{(0)}=\mathcal{A}_{1} \Delta \left(\mathcal{A}_{2}\ast \left([n]\setminus \mathcal{Q}\right)\right) \Delta \cdots \Delta \left(\mathcal{A}_{\nu_{q}(A)+1}\ast \left([n]\setminus \mathcal{Q}\right)^{\ast \nu_{q}(A)}\right),\]
and then
\[A_{q}^{(0)}\ast \left([n]\setminus \mathcal{Q}\right)=\bigdelta_{i=1}^{\nu_{q}(A)+1} \mathcal{A}_{i}\ast \left([n]\setminus \mathcal{Q}\right)^{\ast i}.\]
Suppose that there is some $i$ such that $\mathcal{A}_{i}\ne \emptyset$ and $\beta(i)\ge 2$. Then if $n\ge 6561434682105162$ then $n^{1/5}\ge 40\log n$ and so by Lemma \ref{A_chain_lower_bound},
\[\left|A_{q}^{(0)}\ast \left([n]\setminus \mathcal{Q}\right)\right|\ge \frac{n^{2}}{25(\log n)^{2}}> n,\]
which implies \eqref{add_q_induction_step}. So we may suppose that $\mathcal{A}_{i}\ne \emptyset$ only if $i$ is a power of $2$. Let $u_{1}<\dots<u_{r}$ be those $u_{i}$ such that $\mathcal{A}_{u_{i}}\ne \emptyset$ and
\[\mathcal{B}_{i}=\mathcal{A}_{u_{i}}.\]

When $r=1$, let $\mathcal{B}=\mathcal{B}_{1}$ and $u=u_{1}$. For all $0\le i\le \nu_{q}(A)$, we have 
\[A_{q}^{(i)}=\mathcal{B}\ast \left([n]\setminus \mathcal{Q}\right)^{\ast (\nu_{q}(A)-i)}.\]
It follows that $\gcd(\mathcal{B})=\gcd(A)=1$. From \eqref{large_n_sliceformula}, we have
\[\left|A\ast \left([n]\setminus\mathcal{Q}'\right)\right|=|\mathcal{B}|+\left|\mathcal{B}\ast \left([n]\setminus \mathcal{Q}\right)^{\ast u}\right|.\]
Recalling that $u$ is a power of $2$, applying Lemma \ref{power_2_power}, there is some nonempty $\mathcal{B}'\subseteq \mathcal{B}$ such that $|\mathcal{B}\ast \left([n]\setminus \mathcal{Q}\right)^{\ast u}|\ge |\mathcal{B}'\ast \left([n]\setminus \mathcal{Q}\right)|$, so we may suppose without loss of generality that $u=1$. We apply  Lemma \ref{A_star_n_minus_T_bound}, with $\eps=1/5$. We have $|\mathcal{B}|\le n^{1/5}+1$ and we can check that for $n\ge 2\cdot (3^{80}-321)$ this implies that $n\ge (8004\log n)^{5}> (40020\log |\mathcal{B}|-20010\log 2)^{5}$ and $n\ge 4480^{5}$. We therefore have
\[\left|\mathcal{B}\ast \left([n]\setminus \mathcal{Q}\right)\right|\ge n\]
unless $\mathcal{B}=\{1\}$ or $\{1,2\}$. If $\mathcal{B}=\{1\}$ then $|A\ast ([n]\setminus \mathcal{Q}')|= n+1-|\mathcal{Q}|=n-|\mathcal{Q}'|$. If $\mathcal{B}=\{1,2\}$ then $\mathcal{B}\ast ([n]\setminus \mathcal{Q})=([n]\setminus \mathcal{Q})\Delta (2\cdot ([n]\setminus \mathcal{Q}))$. It is easily seen that this set has size at least $n-2|\mathcal{Q}|$ and then $|A\ast ([n]\setminus \mathcal{Q}')|\ge n+2-2|\mathcal{Q}|=n-2|\mathcal{Q}'|$.

Suppose now that $r\ge 2$. We have
\[A_{q}^{(0)}\ast \left([n]\setminus \mathcal{Q}\right)=\bigdelta_{i=1}^{r} \mathcal{B}_{i}\ast \left([n]\setminus \mathcal{Q}\right)^{\ast u_{i}}.\]
Then by Lemma~\ref{A_star_n_minus_T_bound}  and Lemma~\ref{power_2_power}, for any $i$, we have
\[\left|\mathcal{B}_{i}\ast \left([n]\setminus \mathcal{Q}\right)^{\ast u_{i}}\right|\ge n-2|\mathcal{Q}|.\]
By Lemma \ref{star_power_2_intersection}, letting $m=\sum_{i=1}^{r}|\mathcal{B}_{i}|$, we have
\begin{align*}\left|\bigdelta_{i=1}^{r} \mathcal{B}_{i}\ast \left([n]\setminus \mathcal{Q}\right)^{\ast u_{i}}\right|&\ge r\left(n-2|\mathcal{Q}|\right)-2\sum_{1\le i<j\le r}|\mathcal{B}_{i}|\cdot|\mathcal{B}_{j}|n^{2^{i-j}}\\ &\ge r\left(n-2|\mathcal{Q}|\right)-2 m^{2}n^{1/2}.\end{align*}
The worst case is when $r=2$. Recalling that $m<n^{1/5}+1\le 2n^{1/5}$ and $|\mathcal{Q}|\le n^{4/5}$, we get a lower bound of $n(2-4n^{-1/5}-8n^{-1/10})$,  which implies \eqref{add_q_induction_step} when $n\ge 12^{10}$. So in all cases, \eqref{add_q_induction_step} holds, which completes the proof of the theorem.
\end{proof}

\section{Acknowledgements}
Both authors were supported by the National Research, Development and Innovation Office NKFIH (Excellence program, Grant Nr. 153829). PPP was also supported by the National Research, Development and Innovation Office NKFIH (Grant Nr. K146387) and by the Institute for Basic Science (IBS-R029-C4).

\end{document}